\documentclass[11pt]{article}

\def\0{\bf 0}
\def\1{\bf 1}

\def\s{\sigma}

\def\E{{\bf E}}

\usepackage{amsmath, amsfonts, amssymb, amsthm}
\usepackage{bbm}
\usepackage{xcolor}
\usepackage[applemac]{inputenc}
\usepackage[english]{babel}
\usepackage{amsthm}

\usepackage{mathrsfs}

\newtheorem{theorem}{Theorem}[section]

\newtheorem{prop}{Proposition}[section]

\newtheorem{lem}{Lemma}[section]

\newtheorem{rmk}{Remark}[section]
\newtheorem{corollary}{Corollary}[section]

\newtheorem{assumption}{Assumption}[section]

\begin{document}

\title{\bf Optimal variable selection and adaptive noisy Compressed Sensing}
\author{Ndaoud, M.$^{1}$ and Tsybakov, A.B.$^{1}$\\
{\small $^1$ CREST (UMR CNRS 9194), ENSAE,
5, av. Henry Le Chatelier, 91764 Palaiseau, France}
}

\maketitle


\begin{abstract}
In the context of high-dimensional linear regression models, we propose an algorithm of exact support recovery in the setting of noisy compressed sensing where all entries of the design matrix are independent and identically distributed standard Gaussian. This algorithm achieves the same conditions of exact recovery as the exhaustive search (maximal likelihood) decoder, and has an advantage over the latter of being adaptive to all parameters of the problem and computable in polynomial time. The core of our analysis consists in the study of the non-asymptotic minimax Hamming risk of variable selection. This allows us to derive a procedure, which is nearly optimal in a non-asymptotic minimax sense. Then, we develop its adaptive version, and propose a robust variant of the method to handle datasets with outliers and heavy-tailed distributions of observations. The resulting polynomial time procedure is near optimal, adaptive to all parameters of the problem and also robust.
 \end{abstract}


 \section{Introduction}

\subsection{Statement of the problem}

Assume that we have the vector of measurements $Y\in\mathbb{R}^{n}$ satisfying 
\begin{equation}\label{eq:def}
     Y = X\beta + \sigma \xi 
\end{equation}
where $X\in\mathbb{R}^{n \times p}$ is a given design or sensing matrix, $\beta \in\mathbb{R}^{p}$ is the unknown signal, and $\sigma>0$. In this paper, we mostly focus on the setting where all entries of $X$ are independent identically distributed (i.i.d.) standard Gaussian random variables and the noise $\xi \sim \mathcal{N}\left(0,\mathbb{I}_{n} \right)$ is a standard Gaussian random vector  independent of $X$. Here, $\mathbb{I}_{n}$ denotes the $n\times n$ identity matrix. This setting is typical for noisy compressed sensing, cf. references below. We will also consider extensions to sub-Gaussian design $X$ and to noise $\xi$ with  heavy-tailed distribution. 

 In this paper, one of the main problems that we are interested in consists in recovering the support of $\beta$, that is the set $S_{\beta}$ of non-zero components of $\beta$. For an integer $s\le p$, we assume that $\beta$ is $s$-sparse, that is it has at most $s$ non-zero components.  We also assume that these components cannot be arbitrarily small. This motivates us to define the following set $\Omega^{p}_{s,a}$ of $s$-sparse vectors: 
$$  \Omega^{p}_{s,a} = \left\{\beta \in \mathbb{R}^{p}:\quad  |\beta|_{0} \leq s \quad \text{and} \quad   |\beta_{i}| \geq a, \quad \forall i \in S_{\beta} \right\}, $$
where $a>0$, $\beta_{i}$ are the components of $\beta$ for $i=1,\dots ,p,$ and $|\beta|_{0}$ denotes the number of non-zero components of $\beta$. We consider the problem of variable selection stated as follows: Given the observations $(X,Y)$, estimate the binary vector
$$ \eta_{\beta} = (\mathbf{1}\{\beta_{1} \neq 0\},\dots,\mathbf{1}\{\beta_{p} \neq 0\}), $$
where $\mathbf{1}\{\cdot\}$ denotes the indicator function.
In order to estimate $\eta_{\beta}$ (and thus the support $S_{\beta}$), we define a {\it selector} $\hat{\eta} = \hat{\eta}(X,Y)$ as a measurable function of the observations $(X,Y)$ with values in $\{0,1\}^{p}$. The performance of selector $\hat{\eta}$ is measured by the maximal risks
$$  \underset{\beta \in \Omega^{p}_{s,a}}{\sup}\mathbf{P}_{\beta}\left( \hat{\eta} \neq \eta_{\beta} \right) \qquad \text{and} \qquad   \underset{\beta \in \Omega^{p}_{s,a}}{\sup}\mathbf{E}_{\beta}\left| \hat{\eta} - \eta_{\beta} \right|$$
where $|\hat{\eta} - \eta_{\beta}|$ stands for the Hamming distance between $\hat{\eta}$ and $\eta_{\beta}$, $\mathbf{P}_{\beta}$ denotes the joint distribution of $(X,Y)$ satisfying (\ref{eq:def}), and $\mathbf{E}_{\beta}$ denotes the corresponding expectation. We say that a selector $\hat{\eta}$ achieves {\it exact support recovery} with respect to one of the above two risks if 
\begin{equation}\label{eq:exact:supp}  \underset{p\to \infty}{\lim} \underset{\beta \in \Omega^{p}_{s,a}}{\sup}\mathbf{P}_{\beta}\left( \hat{\eta} \neq \eta_{\beta} \right) = 0, 
\end{equation}
or
\begin{equation}\label{eq:exact:hamm} \underset{p\to \infty}{\lim} \underset{\beta \in \Omega^{p}_{s,a}}{\sup}\mathbf{E}_{\beta}\left| \hat{\eta} - \eta_{\beta} \right| = 0, 
\end{equation}
where the asymptotics are considered as $p \to \infty$ when all other parameters of the problem (namely, $n$, $s$, $a$, $\sigma$) depend on $p$ in such a way that $n = n(p) \to \infty$. In particular, the high-dimensional setting with $p\ge n$ is covered. {In the rest of the paper, we want to characterize sufficient and necessary conditions on the sample size $n$ in order to ensure  \eqref{eq:exact:supp} or \eqref{eq:exact:hamm} hold}. For brevity, the dependence of these four parameters on $p$ will be further omitted in the notation.
Since
$$  \mathbf{P}_{\beta}\left(\hat{\eta} \neq \eta_{\beta}  \right) \leq  \mathbf{E}_{\beta}|\hat{\eta} - \eta_{\beta}|,$$
the property  \eqref{eq:exact:hamm} implies \eqref{eq:exact:supp}. Therefore, we will mainly study the Hamming distance risk.

{\bf Notation.}
In the rest of this paper we use the following notation. For given sequences $a_{n}$ and $b_n$, we say that $a_{n} = \mathcal{O}(b_{n})$ (resp $a_{n} = \Omega(b_n)$) if, for some $c>0$, $a_{n} \leq c b_{n}$ (resp $a_{n} \geq c b_{n}$) for all integers $n$. We write $a_{n} \asymp b_{n} $ if $a_{n} = \mathcal{O}(b_{n})$ and $a_{n}=\Omega(b_{n})$. For ${\bf x},{\bf y}\in \mathbb{R}^{p}$, $\|{\bf x}\|$ is the Euclidean norm of ${\bf x}$, and ${\bf x}^{\top}{\bf y}$ the corresponding inner product. For a matrix $X$, we denote by $X_j$ its $j$th column. For $x,y \in \mathbb{R}$, we denote by $x\vee y$ the maximum of $x$ and $y$, by $\lfloor x \rfloor$ the maximal integer less than $x$ and we set $x_{+}=x \vee 0$. The notation  $\mathbf{1}\{\cdot\}$ stands for the indicator function, and $|A|$ for
the cardinality of a finite set $A$. We denote by $C$ and $c$ positive constants that can differ on different occurences.

\subsection{Related literature}
The literature on support recovery in high-dimensional linear models under sparsity is very rich and its complete overview falls beyond the format of this paper. Here, we outline some of the relevant results in the context of our contribution. 
\begin{itemize}
    \item The existing selectors (also sometimes called decoders) can be split into two main families. The first family consists of polynomial time algorithms, such as selectors based on the Lasso \cite{ZH,wainwrightaLasso}, orthogonal matching pursuit \cite{tropp,zhang,caiOMP} or thresholding \cite{fletcher,joseph}.  The second contains exhaustive search methods, for instance, the Maximum Likelihood (ML) decoder; they are generally not realizable in polynomial time. The ML decoder outputs the support $S_{\hat{\beta}}$ of the least squares solution
    $$ \hat{\beta} \in \underset{\theta: \,|\theta|_{0}=s}{\arg\min}\| Y - X\theta \|, $$
    which is the ML estimator of $\beta$ on the set $\{\beta: \,|\beta|_{0}=s\}$ when the noise is Gaussian. 
     \item The available results are almost exclusively of the form \eqref{eq:exact:supp}, where the asymptotics is considered under various additional restrictions on the behavior of ($n$, $s$, $a$, $\sigma$) as $p\to \infty$. 
     One of the restrictions concerns the magnitude of the noise.
      For $\sigma \asymp 1$, the noise and the entries of the sensing matrix $X$ are of the same order, cf. \cite{fletcher} and 
\cite{wainwrightb}, while \cite{aeron} assumes that $\sigma \asymp \sqrt{n}$, and hence the noise scales largely compared to the signal. 
{Our main results are non-asymptotic bounds on the risk and they can be used in both settings.
We also provide asymptotic corollaries where we assume that $\sigma \asymp 1$.}
    \end{itemize}
    
We now briefly overview results for specific asymptotics, with the emphasis on the {\it phase transition}, that is on the necessary and sufficient conditions of exact recovery.  To the best of our knowledge, they cover only the exact recovery of the type \eqref{eq:exact:supp}.   

In the strong noise regime $\sigma \asymp \sqrt{n}$, \cite{aeron} shows that necessary and sufficient conditions for \eqref{eq:exact:supp} are given by
    $ n = \Omega \left(s\log(\frac{p}{s})\right),$ and $a^{2} = \Omega \left( \log(p-s)  \right)$, and 
     the ML decoder is optimal in the sense that it achieves exact recovery under 
these conditions. However, the ML decoder requires prior knowledge of $s$. In the same regime $\sigma \asymp \sqrt{n}$, \cite{saligrama1} present a polynomial time procedure achieving  
     \eqref{eq:exact:supp} under sub-optimal sufficient conditions
    $n = \Omega \left(s\log(\frac{p}{s})\right)$, and $a^{2} = \Omega \left( (\log p)^{3}  \right). $
    This procedure requires prior knowledge of the threshold~$a$.

        For $\sigma \asymp 1$, which is in fact the general case (equivalent to fixed $\sigma$), the results are different. First, the following necessary condition for exact recovery (in the sense  \eqref{eq:exact:supp}) for any selector is obtained in \cite{wangWain}:
    \begin{equation}\label{wain} 
    n = \Omega\left( \frac{s\log(\frac{p}{s})}{\log(1+s\frac{a^{2}}{\sigma^{2}})} \vee \frac{\log(p-s)}{\log(1+\frac{a^{2}}{\sigma^{2}})} \right). 
    \end{equation}
Based on the analysis of~\eqref{wain}, one might be interested in several regimes for  the
signal-to-noise ratio (SNR)  $a/ \sigma$. In the regime $a/ \sigma= \Omega(1/\sqrt{s})$, 
we have $\|\beta\|^{2} = \Omega(\sigma^{2})$. This can be interpreted as the fact that the total signal is more powerful than noise. 
On the other hand, the condition $a/ \sigma= \Omega(1)$ corresponds to a regime where each 
signal component is more powerful than the noise. In \cite{rad}, it is shown that, under the restrictions  $a/ \sigma= \mathcal{O}(1)$ and $a/\sigma =\Omega( 1/\sqrt{s})$ on the signal-to-noise ratio $a/ \sigma$, the ML decoder is optimal in the sense that it  achieves \eqref{eq:exact:supp}  under the necessary condition \eqref{wain}. Note that the second term in  \eqref{wain} satisfies
\begin{equation}\label{wain0}
\frac{\log(p-s)}{\log(1+\frac{a^{2}}{\sigma^{2}})} \asymp \frac{\sigma^{2}\log(p-s)}{a^{2}} \quad \text{for }  a/ \sigma= \mathcal{O}(1).
\end{equation} 
In the general case,  that is with no restrictions on the joint behavior of $s$, $\sigma$ and $a$, the following sufficient condition for the ML decoder to achieve exact recovery \eqref{eq:exact:supp} is given in \cite{wainwrightb}: 
         \begin{equation}\label{wain1} n = \Omega\left(s\log\Big(\frac{p}{s}\Big) \vee \frac{\sigma^{2}\log(p-s)}{a^{2}} \right). \end{equation}
One can check that, for   $a/\sigma =\mathcal{O}( 1/\sqrt{s})$, the second terms in  \eqref{wain} and in \eqref{wain1} are dominant, while for $a/\sigma =\Omega( 1)$, the first terms are dominant. These remarks and \eqref{wain} - \eqref{wain1} lead us to the following table of phase transitions for exact recovery in the sense of \eqref{eq:exact:supp}. We recall that this table, as well as the whole discussion in this subsection, deal only with the setting where both $X$ and $\xi$ are Gaussian.

\vspace{3mm}

\begin{center}
\begin{tabular}{|c|c|r|}
   \hline
   SNR & Upper bound for ML& Lower bound  \\
   \hline 
   $a/\sigma =\mathcal{O}( 1/\sqrt{s})$ & \multicolumn{2}{c|}{$\frac{\sigma^{2}\log(p-s)}{a^{2}} $}
     \\ 
   \hline
   $a/ \sigma= \mathcal{O}(1)$ and $a/\sigma =\Omega( 1/\sqrt{s})$ & \multicolumn{2}{c|}{$\frac{s\log\left(\frac{p}{s}\right)}{\log\left(1+s\frac{a^{2}}{\sigma^{2}}\right)} \vee \frac{\log(p-s)}{\log\left(1+\frac{a^{2}}{\sigma^{2}}\right)}$} \\
   \hline
   $a/\sigma =\Omega( 1)$ & $s\log\left(\frac{p}{s}\right)$ & $\frac{s\log({p}/{s})}{\log\left(1+s a^{2}/\sigma^{2}\right)}$ \\
   \hline
\end{tabular}
\vspace{2mm}

\footnotesize{Table 1. Phase transitions in Gaussian setting: upper and lower bounds on the 
sample size $n$.}
  \end{center}
  
\vspace{3mm}
  
It remains an open question what is the exact phase transition for $a/\sigma =\Omega( 1)$. We also note that, in the zone $a/ \sigma= \mathcal{O}(1)$, the exact phase transitions in this table are attained by the ML decoder, which is not computable in polynomial time and requires the knowledge of $s$. 
         Known polynomial time algorithms are shown to be optimal only in the regime $a/\sigma =\mathcal{O}\left( 1/\sqrt{s}\right)$. In \cite{fletcher}, it is shown that Lasso is sub-optimal compared to the ML decoder. For the regime $a^{2}/\sigma^2=\mathcal{O}\left(\frac{\log(s)}{s}\right)$ and $s\asymp p$, the ML decoder requires $n=\Omega(p)$ observations to achieve exact recovery, while polynomial time algorithms require $n=\Omega(p\log(p))$. In this regime, the ML decoder is optimal, cf. 
Table 1.  In the regime of $a/\sigma = \Omega(1)$, { it is conjectured that there exists an algorithmic gap making the problem of exact recovery hard whenever the 
sample size satisfies $n\leq c\sigma^{2}s\log(p)$ , for some sufficiently small constant $c>0$ \cite{saligrama3}}.

Variable selection algorithms based on techniques from sparse graphs theory such as sparsification of the Gram matrix $X^{\top}X$ are suggested in \cite{jin1}, \cite{jin2} and \cite{jin3}. In those papers, phase transitions are derived for the asymptotics where the sparsity $s$ and the sample size $n$ scale as power functions of the dimension $p$.  In general, sufficient conditions for the ML decoder are less restrictive than conditions obtained for known polynomial time algorithms. A more complete overview of necessary and sufficient conditions for exact recovery defined in the form \eqref{eq:exact:supp} for different models can be found in \cite{aksoylar}.

\subsection{Contributions}\label{sec:contr}

The main contributions of this paper can be summarized as follows:
\begin{itemize}
\item based on a non-asymptotic study of the minimax Hamming risk, we propose a polynomial time method that achieves exact recovery with respect to both criteria \eqref{eq:exact:supp} and \eqref{eq:exact:hamm} under the same sufficient conditions \eqref{wain1} as  the ML decoder;
\item we develop an adaptive version of this method, which does not depend on the parameters $a, s$ and $\s$ and shares analogous properties; we also extend it to sub-Gaussian $X$ and $\xi$;
\item we propose a robust version of this method to handle data sets with outliers and heavy-tailed distributions of observations.
\end{itemize}
An open question stated in \cite{fletcher} is whether a computationally tractable algorithm can achieve a scaling similar to the ML decoder. This paper answers the question positively under rather general conditions. 

The selector $\hat{\eta}$ that we suggest here is defined by a two step algorithm {based on two subsamples. Using the first subsample we estimate $\beta$ by $\hat{\beta}$, in a way that provides a control on $\| \hat{\beta} - \beta\|$ with high probability. Although many methods can be used (e.g., the LASSO) we choose to consider the Square-Root SLOPE estimator, as it is adaptive to parameters $s,\sigma$ and $a$}.
At the second step, the components of $\hat{\eta}$ are obtained by thresholding of debiased estimators of the components of $\beta$ based on the preliminary estimator $\hat{\beta}$. 

We now proceed to the formal definition of this selection procedure. Split the sample $(X_{i},Y_{i})$, $i=1,\dots,n$, into two subsamples ${\mathcal D}_1$ and ${\mathcal D}_2$ with respective sizes $n_{1}$ and $n_{2}$, such that $n=n_{1} + n_{2}$. For $k=1,2,$ denote by $(X^{(k)},Y^{(k)})$ the corresponding submatrices $X^{(k)} \in \mathbb{R}^{n_{k} \times p}$ and subvectors $Y^{(k)} \in \mathbb{R}^{n_{k}}$. The Square-Root SLOPE estimator based on the first subsample $(X^{(1)},Y^{(1)})$ is defined as follows. Let $\lambda \in \mathbb{R}^{p}$
be a vector of tuning parameters
$$\lambda_{j} = A\sqrt{\frac{\log(\frac{2p}{j})}{n}}, \quad \quad j=1,\dots ,p,$$
for a constant $A>0$ large enough. For example, if $\xi$ is a standard Gaussian random vector, it suffices to take $A
>16+4\sqrt{2}$, cf. \cite{deru}. 
For any $\beta \in \mathbb{R}^{p}$, let $(\beta^{*}_{1},\dots,\beta_{p}^{*})$
be the non-increasing rearrangement of $|\beta_1|, \dots , |\beta_p|$. Consider
$$ |\beta|_{*} = \sum_{j=1}^{p} \lambda_{j}\beta^{*}_{j}, \quad \beta \in \mathbb{R}^{p},$$
which is a norm on $\mathbb{R}^{p}$, cf., e.g.,  \cite{candes}. The Square-Root SLOPE estimator is a solution of the convex minimization problem
\begin{equation}\label{def_sqr_slope} 
\hat{\beta} \in \underset{\beta \in \mathbb{R}^{p}}{\arg\min} \left( \frac{ \| Y^{(1)} - X^{(1)}\beta\|}{\sqrt{n_{1}}} + 2|\beta|_{*}\right).
\end{equation}
Note that this estimator does not depend on the parameters $s$, $\sigma$, and $a$. Details about the computational aspects and statistical properties of the Square-Root SLOPE estimator can be found in \cite{deru}.

The suggested selector is defined as a binary vector
\begin{equation}\label{eq:selec1}
    \hat{\eta}(X,Y) = \left( \hat{\eta}_{1}(X,Y),\dots,\hat{\eta}_{p}(X,Y) \right) 
\end{equation}
with components 
\begin{equation}\label{eq:selec2}
    \hat{\eta}_{i}(X,Y) = \mathbf{1}\left\{  \frac{  \left|X^{(2)\top}_{i}\left( Y^{(2)} - \sum_{j \neq i}X^{(2)}_{j}\hat{\beta}_{j}\right)\right|}{\|X_{i}^{(2)} \|} > t(X^{(2)}_{i}) \right\}
\end{equation}
for $i=1,\dots,p$, where $X^{(2)}_{i}$ denotes the $i$th column of matrix $X^{(2)}$. The threshold $t(\cdot)$ in \eqref{eq:selec2} will be defined by different expressions, with a basic prototype of the form
\begin{equation}\label{t} 
t(u) = t_{\sigma}(u) = \frac{a\|u\|}{2} + \frac{\sigma^{2}\log(\frac{p}{s}-1)}{a \|u\|}, \qquad \forall u \in \mathbb{R}^{n_2}.
\end{equation}
The selector \eqref{eq:selec1} - \eqref{eq:selec2} is the core procedure of this paper. {As explained below, this choice is motivated by a reduction that transforms the original support estimation problem to support estimation in a 
sparse mean model. The latter is solved in an optimal way by a thresholding procedure}. We show 
that the selector \eqref{eq:selec1} - \eqref{eq:selec2} improves upon known sufficient conditions of exact recovery for  methods realizable in polynomial time.  We also show that it can be turned into a completely adaptive procedure (once the sufficient conditions are fulfilled) by suitably modifying the definition \eqref{t}  of the threshold.  
Another advantage is that the selector \eqref{eq:selec1} - \eqref{eq:selec2}  can be generalized to sub-Gaussian design matrices $X$  and to heavy-tailed noise.

Section~\ref{sec:nas} is devoted to the study of non-asymptotic minimax Hamming distance risk. Specifically, Theorem~\ref{th:inf} provides a minimax lower bound for any selector, and plays a central role in this article since it is instrumental in motivating the selector  \eqref{eq:selec1} - \eqref{eq:selec2}. Theorem~\ref{thm:suf1} is the counterpart of Theorem~\ref{th:inf}, where we show that this selector is nearly optimal in a minimax sense. Both theorems involve the quantities denoted by $\psi_+$ and $\psi$, { that are specific to the minimax risk of variable selection in the normal means model}, whose behavior is somewhat complicated. In Section~\ref{sec:phase} we consider different regimes and study the behavior of these quantities, which in turn highlights the presence of interesting phase transitions.
Section~\ref{sec:adap} is devoted to adaptivity to all parameters of the setting, while in Section~\ref{sec:sub} we show how to extend all previous results to sub-Gaussian $X$and $\xi$. Finally, in Section~\ref{sec:rob}, we give a robust version of our procedure when the noise $\xi$ is heavy-tailed and the data are corrupted by arbitrary outliers.

\section{Non-asymptotic bounds on the minimax risk}\label{sec:nas}

Here, as well as in Sections \ref{sec:phase} and \ref{sec:adap}, we assume that all entries of $X$ are i.i.d. standard Gaussian random variables and the noise $\xi \sim \mathcal{N}\left(0,\mathbb{I}_{n} \right)$ is a standard Gaussian random vector  independent of $X$.

In this section, we present a non-asymptotic minimax lower bound on the Hamming risk of arbitrary selectors as well as non-asymptotic upper bounds for the two risks of the selector \eqref{eq:selec1} - \eqref{eq:selec2}. 
In several papers, lower bounds are derived using the Fano lemma in order to get necessary conditions of exact support recovery, i.e., the convergence 
of the minimax risk  to~0. However, they do not give information about the rate of convergence.  Our first aim in this section is to obtain an accurate enough lower bound characterizing the rate. The Fano lemma is too rough for this purpose and we use instead more refined techniques based on explicit Bayes risk calculation.
Set
$$ \psi_{+}\left(n,p,s,a,\sigma \right) = \left(p-s\right)\mathbf{P}\left( \sigma \varepsilon > t\left( \zeta \right) \right) + s \mathbf{P}\left( \sigma \varepsilon \ge a\|\zeta\| - t\left( \zeta \right) \right),$$
where $\varepsilon$ is a standard Gaussian random variable, $\zeta \sim \mathcal{N}\left(0,\mathbb{I}_{n}\right)$ is a standard Gaussian random vector in $\mathbb{R}^{n}$ independent of $\varepsilon$, and $t(\cdot)$ is defined in \eqref{t}. 

The following minimax lower bound holds.
\begin{theorem}\label{th:inf}
For any $a>0$, $\sigma>0$ and any integers $n,p,s$ such that $s<p$ we have 
$$ \forall s'\in(0,s], \quad \underset{\tilde{\eta}}{\inf}\underset{\beta \in \Omega^{p}_{s,a}}{\sup} \mathbf{E}_{\beta}|\tilde{\eta} - \eta_{\beta}| \geq \frac{s'}{s}\left( \psi_{+}(n,p,s,a,\sigma) - 4se^{-\frac{(s-s')^{2}}{2s}}\right), $$
where $\inf_{\tilde{\eta}}$ denotes the infimum over all selectors $\tilde{\eta}$.
\end{theorem}
The proof of this theorem is given in the Appendix.
It relies on a reduction to {the normal means model} that allows us to bound from below the component-wise Bayes risk. Achieving the minimal value of the risk for each component leads to an equivalent of the oracle (non-realizable) selector $\eta^{*}$ with components
\begin{equation}\label{eq:selec3}
    \eta^{*}_{i}(X,Y) = \mathbf{1}\left\{ \frac{ X^{\top}_{i}\left( Y - \sum_{j \neq i}X_{j}\beta_{j}\right)}{\|X_{i} \|} > t(X_{i}) \right\}, \quad i=1,\dots,p,
\end{equation}
where $t(\cdot)$ is the threshold defined in \eqref{t}.
{Clearly, the oracle selector $\eta^*$ is not realizable since it depends on the unknown $\beta$. We do not know the rest of the components of $\beta$ when we try to recover its $i$th component.} This oracle selector has a structure similar to \eqref{eq:selec2}. It selects the components by thresholding  the random variables  
\begin{equation}\label{eq:selec3a}
\frac{ X_{i}^{\top}\left(Y - \sum_{j\neq i}^{}X_{j}\beta_{j}\right)}{\|X_{i}\|^{2}}.  
\end{equation}
This motivates the method that we are proposing.
Note that, under the model \eqref{eq:def}, the random variable \eqref{eq:selec3a} has the same distribution as
$$ \beta_{i} + \sigma \frac{\varepsilon_i}{\|X_{i}\|}, $$
where $\varepsilon_i$ is a standard Gaussian random variable independent of $\|X_{i}\|$.  In simple words, the random 
variable \eqref{eq:selec3a} can be viewed as the value of $\beta_{i}$ plus noise, so that 
thresholding this random variable is a meaningful way to proceed for selection. Moreover, conditionally on the design $X$, we are in the framework of variable selection in the normal means model, where the techniques developed in \cite{butucea} can be 
applied to obtain sharp lower bounds on the risks.

Since the sensing matrix $X$ is assumed Gaussian with i.i.d. entries, it is straightforward to see that $ \sum_{j\neq i}^{}X_{j}\beta_{j} $ is a zero-mean Gaussian random variable with variance not greater than $\| \beta \|^{2}$. Hence we can consider this term as an additive noise, but the fact that we cannot control $\| \beta \|$ means that the variance of the noise is also not controlled. In order to get around this drawback, we plug in an estimator $\hat\beta$ instead of $\beta$ in the oracle expression. This motivates the two-step selector defined in \eqref{eq:selec1} - \eqref{eq:selec2}. At the first step, we use the Square-Root SLOPE estimator $\hat\beta$  based on the subsample~${\mathcal D}_1$. 
We have the following bound on the $\ell_2$ error of the Square-Root SLOPE estimator.

\begin{prop}\label{prop:1}
Let $\hat\beta$ be the Square-Root SLOPE estimator  defined in Section \ref{sec:contr} with constant $A\geq 16 +4\sqrt{2}$. There exist positive constants $C_{0}, C_1$ and $C_2$ such that for all $\delta \in (0,1]$ and $n_1 > \frac{C_{0}}{\delta^2}s\log{\left(\frac{ep}{s}\right)}$ we have
$$ \underset{|\beta|_{0}\leq s}{\sup} \mathbf{P}_{\beta}\left( \|\hat{\beta} - \beta\| \geq  \delta \sigma \right) \leq C_1 \Big(\frac{s}{2p}\Big)^{C_2s}.$$
\end{prop}
This proposition is a special case of Proposition \ref{prop:subGaussian} below. 
The fact that it is enough to take $A
>16+4\sqrt{2}$ when $\xi$ is a standard Gaussian random vector is proved in \cite{deru}.

In what follows, for the sake of readability, we will write $X$ and $Y$ instead of $X^{(2)}$ and $Y^{(2)}$ since we will condition on the first subsample~${\mathcal D}_1$ and only use the second subsample~${\mathcal D}_2$ in our argument. We only need  to remember that $\hat{\beta}$ is independent from the second sample of size $n_{2}$. With this convention, definition \eqref{eq:selec2} involves now the random variables
\begin{equation}\label{alphai0} \alpha_{i}:=\frac{ X_{i}^{\top}\left(Y - \sum_{j\neq i} X_{j}\hat{\beta}_{j}\right)}{\|X_{i}\|} = \beta_{i}\|X_{i}\| + \frac{1}{\|X_{i}\|}X_{i}^{\top}\Big( \sum_{j\neq i}X_{j}(\beta_{j} - \hat{\beta}_{j})+\sigma \xi \Big)
\end{equation}
for $i=1,\dots,p$. Conditionally on $\hat{\beta}$ and $X_i$, the variable $\alpha_{i}$ has the same distribution as 
\begin{equation}\label{alphai} 
\beta_{i} \|X_{i}\| + \Big(\sigma^{2} + \sum_{j\neq i}|\beta_{j} - \hat{\beta_{j}}|^{2}\Big)^{\frac{1}{2}} \varepsilon,
\end{equation}
  where $\varepsilon$ is a standard Gaussian random variable. Hence, considering $\alpha_{i}$ as new observations, we have a conditional normal means model, for which a natural procedure to detect the non-zero components
consists in comparing $\alpha_{i}$ to a threshold. Choosing the same threshold $t(\cdot)$  as in the lower bound of Theorem \ref{th:inf} leads to the selector \eqref{eq:selec1} - \eqref{eq:selec2}.

Consider now a quantity close to $\psi_{+}$ given by the formula
$$ \psi\left(n,p,s,a,\sigma \right) = \left(p-s\right)\mathbf{P}\left(  \sigma \varepsilon > t\left(\zeta \right) \right) + s \mathbf{P}\left( \sigma  \varepsilon > \left(a\|\zeta\| - t\left(\zeta \right)\right)_{+} \right)$$
where $t(u) = t_{\sigma}(u)$ is defined in \eqref{t}.
Note that
$$\psi\left(n,p,s,a,\sigma \right) \leq \psi_{+}\left(n,p,s,a,\sigma \right).$$ We have the following upper bound for the minimax risks of the selector \eqref{eq:selec1} - \eqref{eq:selec2}.

\begin{theorem}\label{thm:suf1}
{Let the assumptions of Proposition \ref{prop:1} be satisfied.} Let $s$ be an integer such that $s\leq p/2$, and let $\hat{\eta}$ be the selector \eqref{eq:selec1} - \eqref{eq:selec2} with the threshold $t(\cdot)=t_{\sigma\sqrt{1+\delta^2}}(\cdot)$ defined in \eqref{t}, with some  $\delta \in (0,1]$.
{{Let the positive constants $C_{0}, C_1$ and $C_2$ be the same as in Proposition \ref{prop:1}} }. For all $n_1 > \frac{C_{0}}{\delta^2}s\log{\left(\frac{ep}{s}\right)}$ we have
$$   \underset{\beta \in \Omega^{p}_{s,a}}{\sup} \mathbf{E}_{\beta}|\hat{\eta} - \eta_{\beta}| \leq 2 \psi(n_{2},p,s,a,\sigma\sqrt{1+\delta^{2}}) + C_1 (s/2)^{C_2s} p^{1-C_2s},$$
and
$$ \underset{\beta \in \Omega^{p}_{s,a}}{\sup}\mathbf{P}_{\beta}\left( \hat{\eta} \neq \eta_{\beta}\right) \leq 2 \psi(n_{2},p,s,a,\sigma\sqrt{1+\delta^{2}}) + C_1(s/2)^{C_2s} p^{-C_2s}.  $$
\end{theorem}

\begin{proof}
 Define the random event
$ \mathbb{A} = \{ \|\hat{\beta}-\beta \| \leq \delta \sigma  \} $, where $\hat{\beta}$ is based on the subsample ${\mathcal D}_1$. 
For any $\beta \in \Omega^{p}_{s,a}$, we have 
\begin{eqnarray*} \mathbf{E}_{\beta}\left[|\hat\eta - \eta_{\beta}| \ | {\mathcal D}_1 \right] &=& \sum_{i:\beta_{i}=0}\mathbf{E}_{\beta}\left[\hat\eta_{i} | {\mathcal D}_1 \right] +
\sum_{i:\beta_{i}\neq 0}\mathbf{E}_{\beta}\left[1-\hat\eta_{i} | {\mathcal D}_1 \right]
\\
&= &
\sum_{i:\beta_{i}=0}\mathbf{P}_{\beta}(|\alpha_{i}|>t(X_i)| {\mathcal D}_1 ) +
\sum_{i:\beta_{i}\neq 0}\mathbf{P}_{\beta}(|\alpha_{i}|\le t(X_i)| {\mathcal D}_1 ). 
\end{eqnarray*}
Here, $t(X_i)\ge 0$ since $s\leq p/2$. Using the fact that, conditionally on $\hat{\beta}$ and $X_i$, the variable $\alpha_{i}$ has the same distribution as \eqref{alphai} we find that, for all $i$ such that $\beta_i=0$,
$$
\mathbf{P}_{\beta}(|\alpha_{i}|>t(X_i)| {\mathcal D}_1)\le \mathbf{P}(\sigma_*|\varepsilon| >t(X_i)| {\mathcal D}_1)=2\mathbf{P}(\sigma_*\varepsilon >t(X_i)| {\mathcal D}_1)
$$
where $\sigma_*= (\sigma^{2}+ \|\hat{\beta}-\beta\|^{2})^{1/2}$ and $\varepsilon$ is a standard Gaussian random variable independent of $\|X_i\|$. An analogous argument and the fact that $|\beta_i|\ge a$ for all non-zero $\beta_i$ lead to the bound
$$
\mathbf{P}_{\beta}(|\alpha_{i}|\le t(X_i)| {\mathcal D}_1)\le \mathbf{P}(\sigma_*|\varepsilon| \ge a\|X_i\| - t(X_i)| {\mathcal D}_1)=2\mathbf{P}(\sigma_*\varepsilon \ge (a\|X_i\| - t(X_i))_+| {\mathcal D}_1)
$$
valid for all $i$ such that $\beta_i\ne 0$, where we have used the fact that $|\alpha_{i}|\geq a\|X_{i}\| - |\alpha_{i} - \beta\|X_{i}\||$. Therefore, 
\begin{equation}\label{15} 
\mathbf{E}_{\beta}\left[|\hat\eta - \eta_{\beta}| \ | {\mathcal D}_1 \right]\le 2(p-s)\mathbf{P}(\sigma_*\varepsilon >t(\zeta)| {\mathcal D}_1) + 2s\mathbf{P}(\sigma_*\varepsilon \ge (a\|\zeta\| - t(\zeta))_+| {\mathcal D}_1),
\end{equation}
where $\zeta \sim \mathcal{N}\left(0,\mathbb{I}_{n_{2}}\right)$ is a standard Gaussian random vector in $\mathbb{R}^{n_{2}}$ independent of $\varepsilon$. { Note that the right hand side of \eqref{15} is equal to $2\psi(n_{2},p,s,a,\sigma^{*})$ where $\sigma^{*}$ is random and depends on $\mathcal{D}_{1}$}. Using this bound on the event $\mathbb{A}$ and taking expectations with respect to ${\mathcal D}_1$ yields
$$\mathbf{E}_{\beta}|\hat\eta - \eta_{\beta}| \leq 2 \psi(n_{2},p,s,a,\sigma\sqrt{1+\delta^{2}}) + 2p \mathbb{P}\left(\mathbb{A}^{c}\right), $$
{where we have taken into account that $2\psi(n_{2},p,s,a,\sigma^{*}) \leq 2p$.}
For $\mathbf{P}_{\beta}\left( \hat{\eta} \neq \eta_{\beta}\right)$, we get an analogous result using a union bound on the event $\mathbb{A}$. The factor $p$ in the second term disappears in this case. %
The theorem follows by applying Proposition \ref{prop:1}.
\end{proof}

\begin{rmk}
As we will see in the next section, the term $p\big(\frac{s}{2p}\big)^{C_2s}$ is small compared to $\psi$ for large $p$. Hence, $\psi$, or the close quantity $\psi_+$, characterize the main term of the optimal rate of convergence. Uniformly on $\Omega^{p}_{s,a}$, no selector can reach a better rate of the minimax risk in the asymptotics. The discrepancy between the upper and lower bounds comes from increasing the sample size by $n_{1}$, in order to estimate $\beta$ (in the upper bound, the first argument of $\psi_+$ is the smaller sample size $n_2<n$, which makes $\psi_+$ greater), and a higher variance $\sigma^{2}(1+\delta^{2})$, even if we can make it very close to $\sigma^{2}$ by choosing $\delta$.   
\end{rmk}

\begin{rmk} Our choice of Square-Root SLOPE estimator $\hat \beta$ is motivated by the fact that it achieves the optimal rate of $\ell_2$ estimation adaptively to $s$ and $\sigma$, which will be useful in Section \ref{sec:adap}. Since  in this section  we do not consider adaptivity issues, we can also use as $\hat \beta$ the LASSO estimator with regularization parameter depending on both $s$ and $\sigma$ or the SLOPE estimator, for which the regularization parameter depends $\sigma$ but not on $s$. Indeed, it 
follows from \cite{bellec} that the conclusion of Proposition \ref{prop:1} holds when $\hat \beta$ is such a LASSO or a SLOPE estimator. Thus, Theorem \ref{thm:suf1} remains valid for these two estimators as well. 
\end{rmk}

\begin{rmk}
{
The sample splitting in our procedure grants independence between the two steps. In practice, sample splitting can be avoided through aggregation or iterative algorithms. Theoretical properties of such alternatives fall beyond the scope of this paper. 
}
\end{rmk}

The values $\alpha_i$ can be viewed as "de-biased" observations in high-dimensional regression.  Other de-biasing schemes can be used, for example, the method considered in Section \ref{sec:rob}.  The most popular de-biasing technique is based on the LASSO. In our context, applying it would mean to replace $\alpha_i$ defined in \eqref{alphai0} by the components $\hat{\beta}^{d}_{i}$
of the vector
$$
\hat{\beta}^{d}= \hat\beta^{L} + \frac{1}{n} X^{\top}\left(Y- X\hat\beta^{L}\right), 
$$
where $\hat\beta^{L}$ is the LASSO estimator (see, for example,
   \cite{Adel} and the references therein).  As in our case, this reduces the initial regression model to the mean estimation model (conditionally on $\hat\beta^{L}$), which is not exactly the normal means model but rather its approximation. Indeed, we may equivalently write
 $$  \hat{\beta}^{d}_{i} = \beta_{i}  + \frac{X_{i}^{\top}}{n}\Big( \sum_{j\neq i}X_{j}(\beta_{j} - \hat{\beta}^{L}_{j})+\sigma \xi \Big)+ \left( 1-\frac{\|X_{i} \|^{2}}{n}\right)( \hat\beta^{L}_{i}-\beta_{i}). $$
The difference from \eqref{alphai0} is in the fact that, conditionally on $\hat\beta^{L}$ and $X_i$,  we have here a bias $\Big( 1-\frac{\|X_{i} \|^{2}}{n}\Big)( \hat\beta^{L}_{i}-\beta_{i})$, and that there is no scaling by the norm of $X_{i}$. 
Note that scaling by the norm $\|X_i\|$ instead of $n$ is crucial in our construction. It allows us to 
obtain in Theorem~\ref{thm:suf1}  the expression for the risk analogous to the lower bound of Theorem~\ref{th:inf}. 

Finally, note that in parallel to our work, a study of a specific type of two-stage algorithms for variable selection in linear models is developed in  \cite{wang17}.  
{The method proposed  in \cite{wang17} consists in estimation through a bridge estimator, 
followed by a thresholding step. The results in \cite{wang17} deal with 
high-dimensional asymptotic setting where the number of observations $n$ grows at the
same rate as the number of predictors $p$, while in the present work  
we develop a non-asymptotic analysis without such a restriction on $n$ and $p$. The 
main aim in \cite{wang17} is to compare variable selection accuracy for different bridge 
estimators used in the first step. Our results and the questions that we address here are significantly different 
since we are interested in necessary and sufficient conditions for variable selection considering
minimax optimality among all possible selectors. }

\section{Phase transition}\label{sec:phase}

Using the upper and lower bounds of Section \ref{sec:nas}, we can now study the phase transition, i.e., the necessary and sufficient conditions on the sample size to achieve exact recovery under the Hamming risk. A first lower bound is given by the following result.
\begin{prop}\label{prop:nec1}
Let $s \geq 6$ and $n \leq \frac{2\sigma^{2}\log\left(\frac{p}{s}-1\right)}{a^{2}} $. There exists an absolute constant $c>0$ such that
$$   \underset{\tilde{\eta}}{\inf}\underset{\beta \in \Omega^{p}_{s,a}}{\sup}\mathbf{E}_{\beta}\left|\tilde{\eta} - \eta_{\beta} \right| \geq \left( c \vee \frac{s}{8}\left( 1 - 16e^{-\frac{s}{8}}\right) \right) ,$$
where $\inf_{\tilde{\eta}}$ denotes the infimum over all selectors $\tilde{\eta}$.
\end{prop}
\begin{proof}
We start by proving a lower bound on the function $\psi_{+}$. We have
$$ \psi_{+}(n,p,s,a,\sigma) \geq s\mathbf{P}\left( \sigma \varepsilon \ge  a\| \zeta\| - t( \zeta ) \right)
\ge s\mathbf{P}(\sigma \varepsilon\ge 0)\mathbf{P}\left(\mathbb{B}\right) = \frac{s}{2} \mathbf{P}\left(\mathbb{B}\right) . $$
where $\mathbb{B} = \left\{ a\| \zeta\| \leq t(\zeta) \right\}$. 
Since a chi-squared random variable with $n$ degrees of freedom has a median smaller than $n$, we get under the conditions stated above that
$$  \mathbf{P}\left( \mathbb{B} \right) = \mathbf{P}\left( \| \zeta\|^{2} \leq \frac{2\sigma^{2} \log(\frac{p}{s}-1)   }{a^{2}} \right) \geq \frac{1}{2}. $$
Therefore, using Theorem \ref{th:inf} we get
$$ \forall s' \in (0,s], \quad \underset{\tilde{\eta}}{\inf}\underset{\beta \in \Omega^{p}_{s,a}}{\sup}\mathbf{E}_{\beta}\left|\tilde{\eta} - \eta_{\beta} \right| \geq s'\left( \frac{1}{4} - 4e^{-\frac{(s-s')^{2}}{2s}} \right). $$
Since $s\geq6$ we have $ 4e^{-s/2} < \frac{1}{4} $. Hence,
$$ \underset{s' \to 0^{+}}{\lim}\left( \frac{1}{4} - 4e^{-\frac{(s-s')^{2}}{2s}} \right) = \frac{1}{4} - 4e^{-s/2} > 0. $$
Thus, there  exists $c>0$ such that 
$$\underset{\tilde{\eta}}{\inf}\underset{\beta \in \Omega^{p}_{s,a}}{\sup}\mathbf{E}_{\beta}\left|\tilde{\eta} - \eta_{\beta} \right| \geq c.  $$
By setting $s'=s/2$, we also get
$$\underset{\tilde{\eta}}{\inf}\underset{\beta \in \Omega^{p}_{s,a}}{\sup}\mathbf{E}_{\beta}\left|\tilde{\eta} - \eta_{\beta} \right| \geq \frac{s}{8}\left( 1 - 16e^{-\frac{s}{8}}\right).  $$
The proposition follows.
\end{proof}
Proposition \ref{prop:nec1} implies that the condition $ n \geq \frac{2\sigma^{2}\log(\frac{p}{s}-1)}{a^{2}}  $ is necessary to achieve exact recovery for the Hamming risk. We give now a more accurate necessary condition for the regime $a = \mathcal{O}(\sigma)$. This regime is the most interesting when we consider the asymptotic setting where $a$ is decreasing.

\begin{theorem}\label{thm:nec1}
{Let $n,a,\sigma,s,p$ such that} $n>\frac{2\sigma^{2}\log\left(\frac{p}{s}-1\right)}{a^{2}}$, $a < \sqrt{2}\sigma$, and $s<p/2$, then there exists an absolute constant $c>0$ such that
$$ \underset{\tilde{\eta}}{\inf}\underset{\beta \in \Omega^{p}_{s,a}}{\sup}\mathbf{E}_{\beta}\left|\tilde{\eta} - \eta_{\beta} \right| \geq  c\sqrt{\frac{s^{7/4}(p-s)^{1/4}}{n\log\left(1+\frac{a^{2}}{4\sigma^{2}}\right)}}\exp\left(-\frac{n}{2}\log\left(1+\frac{a^{2}}{4\sigma^{2}}\right)\right) - 2se^{-\frac{s}{8}},  $$
where $\inf_{\tilde{\eta}}$ denotes the infimum over all selectors $\tilde{\eta}$.
\end{theorem}
The proof of Theorem \ref{thm:nec1} is given in the Appendix.

\begin{corollary}\label{cor:nec}
Let $s \geq 6$, $a < \sqrt{2} \sigma$, and let 
$$ n < (1-\epsilon)\frac{\log(p-s) + 7 \log(s)}{4\log\left(1+\frac{a^{2}}{4\sigma^{2}}\right)}, $$
for some $\epsilon\in (0,1)$. Then, there exists $c>0$ such that 
$$ \underset{p \to \infty}{\liminf} \,\underset{\tilde{\eta}}{\inf}\underset{\beta \in \Omega^{p}_{s,a}}{\sup}\mathbf{E}_{\beta}\left|\tilde{\eta} - \eta_{\beta} \right| \geq c. $$
\end{corollary}
\begin{proof}
If $n \leq  \frac{2\sigma^{2}\log\left(\frac{p}{s}-1\right)}{a^{2}} $, then the result follows from Proposition \ref{prop:nec1}. Now if $n >  \frac{2\sigma^{2}\log\left(\frac{p}{s}-1\right)}{a^{2}} $, then Theorem \ref{thm:nec1} yields
$$ \underset{\tilde{\eta}}{\inf}\underset{\beta \in \Omega^{p}_{s,a}}{\sup}\mathbf{E}_{\beta}\left|\tilde{\eta} - \eta_{\beta} \right|  \geq  c\left( \frac{\left(s^{7/4}(p-s)^{1/4}\right)^{\epsilon}}{(1-\epsilon)\log\left( s^{7/4}(p-s)^{1/4} \right)} \right)^{\frac{1}{2}} - 2se^{-\frac{s}{8}}.   $$
As $1\le s<p$, we have $\underset{p\to\infty}{\lim}s^{7/4}(p-s)^{1/4}=\infty$. The result follows.
\end{proof}
Corollary \ref{cor:nec} implies the following necessary condition for exact recovery under the Hamming risk: 
$$n \geq \frac{\log(p-s) + 7 \log(s)}{4\log\left(1+\frac{a^{2}}{4\sigma^{2}}\right)}.$$
This condition implies that 
$$n \geq \frac{\sigma^{2}(\log(p-s) +7 \log(s))}{a^{2}}.$$ 
The right hand side here is greater than $\frac{2\sigma^{2}\log(p/s-1)}{a^{2}}$, for instance, in the regime $s \ge p^{1/9}$. 
Thus, in this regime, we get a stronger result than the necessary 
condition $n\ge \frac{2\sigma^{2}\log(p/s-1)}{a^{2}}$ of Proposition \ref{prop:nec1}.

We will now show that the upper bound on the minimax risk decreases exponentially with the sample size. This will allow us to show that  the selector \eqref{eq:selec1} - \eqref{eq:selec2} achieves exact recovery under the same conditions as the ML decoder.
\begin{theorem}\label{thm:suf}
{Let the assumptions of Theorem~\ref{thm:suf1} be satisfied} and let $a\le\sigma$. Assume that for some $\delta \in (0,1]$ the following inequalities hold
$$ n_1>\frac{C_{0}}{\delta^2}s\log\left(\frac{ep}{s}\right)  \qquad \text{and} \qquad n_{2} \geq \frac{4\sigma^{2}\log\left(\frac{p}{s}-1\right)}{a^{2}}$$
where $C_0>0$ is the same as in Theorem~\ref{thm:suf1}.
Let $\hat{\eta}$  be the selector as in Theorem \ref{thm:suf1}. Then,
$$ \underset{\beta \in \Omega^{p}_{s,a}}{\sup}\mathbf{E}_{\beta}\left| \hat{\eta} - \eta_{\beta}\right|  \leq  2\sqrt{s(p-s)}\exp\left(-\frac{n_{2}}{2}\log\left(1+\frac{a^{2}}{4\sigma^{2}\left(1+\delta^{2}\right)}\right)\right) +  s e^{-\frac{n_{2}}{24}} + C_1p\Big(\frac{s}{2p}\Big)^{C_2s},$$
and
$$ \underset{\beta \in \Omega^{p}_{s,a}}{\sup}\mathbf{P}_{\beta}\left(  \hat{\eta} \neq \eta_{\beta} \right)  \leq  2\sqrt{s(p-s)}\exp\left(-\frac{n_{2}}{2}\log\left(1+\frac{a^{2}}{4\sigma^{2}\left(1+\delta^{2}\right)}\right)\right) +  s e^{-\frac{n_{2}}{24}} + C_1\Big(\frac{s}{2p}\Big)^{C_2s},$$
where $C_1>0$ and $C_2>0$ are the same as in Theorem~\ref{thm:suf1}.
\end{theorem}

The proof of this theorem is given in the Appendix. It follows from Theorem~\ref{thm:suf1} by bounding $\psi$ from above.

   We can notice that both types of errors decrease exponentially as  the sample size increases to $\infty$.  

\begin{corollary}\label{cor:suf}
Under the conditions of Theorem \ref{thm:suf}, if $a\le\sigma/\sqrt{3}$ and $n_{2}\ge B\frac{\log(p-s) + \log(s)}{\log\left(1+\frac{a^{2}}{4\sigma^{2}\left(1+\delta^{2}\right)}\right)}$ for some $B>1$,  we have
$$ \underset{\beta \in \Omega^{p}_{s,a}}{\sup}\mathbf{P}_{\beta}\left( \hat{\eta} \neq \eta_{\beta} \right)  \leq  3\left(s(p-s)\right)^{\frac{1-B}{2}} + C_1\Big(\frac{s}{2p}\Big)^{C_2s}.  $$
\end{corollary}
\begin{proof}
Since $\log(p-s) + \log(s) \geq  \log\left(\frac{p}{s}-1\right)$, and $\log\left(1 + x\right) \leq x$, we get 
$$ n_{2} \geq \frac{4\sigma^{2}\log\left( \frac{p}{s}-1\right)}{a^{2}}. $$
Hence Theorem \ref{thm:suf} applies. Moreover, since $a\le\sigma/\sqrt{3}$ we also have 
$$ e^{-\frac{n_{2}}{24}} \leq \exp\left(-\frac{n_{2}}{2}\log\left(1+\frac{a^{2}}{4\sigma^{2}\left(1+\delta^{2}\right)}\right)\right).  $$
We can conclude by using the lower bound on  $n_{2}$  and  the inequality $s \leq \sqrt{s(p-s)}$.
\end{proof}
As consequence of the last corollary, sufficient conditions for the selector \eqref{eq:selec1} - \eqref{eq:selec2} with threshold \eqref{t} to achieve exact recovery are as follows
$$  n_1>\frac{C_{0}}{\delta^2}s\log\left(\frac{ep}{s}\right)\quad \text{and}\quad  n_{2} > (1+\epsilon)\frac{\log(p-s) + \log(s)}{\log\left(1+\frac{a^{2}}{4\sigma^{2}\left(1+\delta^{2}\right)}\right)},$$
for some $\delta \in (0,1]$ and $\epsilon >0$. 

Comparing the rate of convergence in Corollary \ref{cor:suf} to the rate for the ML decoder established in \cite{rad}, we notice that they have similar form. Indeed, \cite{rad} proves the bound 
$$  \underset{\beta \in \Omega^{p}_{s,a}}{\sup}\mathbf{P}_{\beta}\left(  \hat{\eta} \neq \eta_{\beta} \right)  \leq  s\bigg(\left(es(p-s)\right)^{-B^{*}} + \left(\frac{s}{e(p-s)}\right)^{B^{*}s}\bigg),  $$
for some $B^{*}>0$ and $s\le p/2$.

It is interesting to compare these conditions with the best known in the literature (where only the risk \eqref{eq:exact:supp}  was studied). 
Using  \eqref{wain0}, we see that, in the zone $a/\sigma = \mathcal{O}(1)$, our sufficient condition for exact recovery
has the form
\begin{equation}\label{suff}
n = \Omega\left(s\log\Big(\frac{p}{s}\Big) \vee \frac{\sigma^{2}\log(p-s)}{a^{2}} \right).
\end{equation}
As follows from the discussion in the Introduction, {this gives the exact phase transition in the zone $a/\sigma = O(1/\sqrt{s})$, while in the zone $a/\sigma = \mathcal{O}(1)$, $a/\sigma=\Omega(1/\sqrt{s})$,} combination of the results of \cite{wangWain} and \cite{rad} shows that the exact phase transition (realized by the ML decoder) is given by
$$ 
    n = \Omega\left( \frac{s\log(\frac{p}{s})}{\log(1+s\frac{a^{2}}{\sigma^{2}})} \vee \frac{\sigma^{2}\log(p-s)}{a^{2}} \right). 
    $$
It remains an open question whether the improvement by the term $\log(1+s\frac{a^{2}}{\sigma^{2}})$ appearing here is achievable by computationally tractable methods. 

Our sufficient condition \eqref{suff} is  the same as for the ML decoder \cite{wainwrightb}, with the advantage that our selector can be computed in polynomial time. Nevertheless, the knowledge of parameters $s,a$ and $\sigma$ is required for the construction. This motivates us to derive, in the next section, adaptive variants of the proposed selector.

\section{ Nearly optimal adaptive procedures}\label{sec:adap}

In this section, we propose three adaptive versions of our selector. The first one assumes that we know only $a$ and do not know $s$ and $\sigma$, the second assumes only the knowledge of $\sigma$, and the third one is completely adaptive to all the parameters.  

We first present the following a tail bound for the Student distribution that will be useful to derive the results.

\begin{lem}\label{lem:ck}
Let $Z$ be a Student random variable with $k$ degrees of freedom. There exist constants $c,C>0$ independent of $k$ such that for all $b\ge  1/\sqrt{k}$ we have
$$  c \frac{(1+b^{2})^{-\frac{k-1}{2}}}{\sqrt{k}b} \leq \mathbf{P}\left(|Z| \geq \sqrt{k}b\right) \leq C \frac{(1+b^{2})^{-\frac{k-1}{2}}}{\sqrt{k}b}. $$
\end{lem}

The proof of this lemma is given in the Appendix.


The Square-Root SLOPE estimator $\hat \beta$ is adaptive to the sparsity parameter $s$ and to the scale parameter $\sigma$. The dependence of the selector $\hat{\eta}$ defined in \eqref{eq:selec1} - \eqref{eq:selec2} on the parameters $s$, $\sigma$ and $a$ only appears in the definition of the threshold $t(\cdot)$. Hence, we will replace it by an adaptive threshold. In this section, we assume that $n$ is an even integer and the sample splitting is done in two subsamples of equal sizes such that $n_{1}=n_{2}=n/2$. In Theorem \ref{thm:suf}, we have shown that the selector $\hat{\eta}$ defined in \eqref{eq:selec1} - \eqref{eq:selec2}  with the threshold  function
\begin{equation}\label{t2}t(u)=\frac{a\|u\|}{2}+ \frac{(1+\delta^2)\sigma^{2}\log\left(\frac{p}{s}-1\right)}{a\|u\|},\quad \forall u \in \mathbb{R}^{n_2},
\end{equation}
achieves nearly optimal conditions of exact recovery.
We now set a new threshold by simply dropping the second term in  \eqref{t2}:
\begin{equation}\label{eq:thres2}
    t(u)=\frac{a\|u\|}{2},\quad \forall u \in \mathbb{R}^{n_2}.
\end{equation}
Then, the procedure becomes adaptive to unknown $s$ and $\sigma$, {but still requires knowledge of $a$}. The phase transition for this procedure is given by the following proposition. 

\begin{prop}\label{prop:ad}
{Let the assumptions of Theorem \ref{thm:suf1} be satisfied, where we relax the sparsity assumption to $s<p$. Let $n$ be an even integer and $2(1\vee 1/C_2)\le s$.} Set $n_{1} = n_{2} = n/2$, and let the threshold $t(\cdot)$ be defined in \eqref{eq:thres2}. Then, the selector $\hat{\eta}$ defined in  \eqref{eq:selec1} - \eqref{eq:selec2} achieves exact recovery under both risks (Hamming and support recovery) if $$n \geq 2\Big(C_{0}s\log\Big(\frac{ep}{s}\Big) \vee \frac{2\log{p}}{\log\left(1+\frac{a^{2}}{8\sigma^{2}}\right)} +1\Big).$$
\end{prop}

\begin{proof}
Following the lines of the proof of Theorem \ref{thm:suf1} and choosing there $\delta = 1$ we get
\begin{equation}\label{eq:aa}  \underset{\beta \in \Omega^{p}_{s,a}}{\sup}\mathbf{E}_{\beta}\left| \hat{\eta} - \eta_{\beta} \right| \leq p\mathbf{P}\left( \sqrt{2}\sigma \left|\varepsilon \right|\ge \frac{a\|\zeta\|}{2} \right)+C_1p\Big(\frac{s}{2p}\Big)^{C_2s},
\end{equation}
where $\varepsilon$ is a standard Gaussian random variable and $\zeta \sim \mathcal{N}\left(0,\mathbb{I}_{n_2}\right)$ is a standard Gaussian random vector in $\mathbb{R}^{n_2}$ independent of $\varepsilon$.
In order to prove exact recovery, we need to show that both terms on the right hand side of \eqref{eq:aa} vanish as $p$ goes to infinity. We first consider the second term. Note that  the function $t\mapsto \big(\frac{t}{2p}\big)^{t}$ is decreasing for $1\le t\le p/2$. Thus, if $2(1\vee 1/C_2) \le s \leq p/2$ we have 
$$  p\Big(\frac{s}{2p}\Big)^{C_2s}\le  p\Big(\frac{1\vee 1/C_2}{p}\Big)^{2} \to 0 \quad {\rm as} \ p\to \infty,$$
while for $ p/2 < s < p$, 
$$ p\Big(\frac{s}{2p}\Big)^{C_2s}\le  p 2^{-C_2 p/2}. $$
Thus, to prove the proposition,  it remains to show that the first term on the right hand side of \eqref{eq:aa} vanishes. Using the independence between $\varepsilon$ and $\zeta$, we have
$$ \mathbf{P}\left( \sqrt{2}\sigma \left|\varepsilon \right|\ge\frac{a\|\zeta\|}{2} \right)=\mathbf{P}\left(|Z|\geq\frac{a\sqrt{n_{2}}}{2\sqrt{2}\sigma}\right), $$
where $Z$ is a Student random variable with $n_{2}$ degrees of freedom. To bound the last probability, we use Lemma \ref{lem:ck}.  Since $\log(1+x) \leq x ,\forall x\geq 0$, the assumption on $n_2$ implies 
$$n_{2} >  \frac{16\sigma^{2}\log{p}}{a^{2}}.$$
In particular, since $p\ge 3$ we have $ \frac{n_{2} a^{2}}{8\sigma^{2}} \geq 1$. Thus, by Lemma \ref{lem:ck},
\begin{align} \nonumber
p\mathbf{P}\left( \sqrt{2}\sigma \left|\varepsilon \right|\ge\frac{a\|\zeta\|}{2} \right)& \leq \frac{Cp\sigma}{a\sqrt{n_{2}}}  \left(1+\frac{{a}^{2}}{8\sigma^{2}}\right)^{-\frac{n_{2}-1}{2}} \\
&= \nonumber
\frac{C\sigma}{a\sqrt{n_{2}}} \exp\left(\log p-\frac{n_{2}-1}{2}\log\left(1+\frac{a^{2}}{8\sigma^{2}}\right)\right)\\
&\leq \frac{C\sigma \sqrt{ \log\left(1+\frac{a^{2}}{8\sigma^{2}}\right)}}{a\sqrt{2\log p}} \label{rhss}
\end{align}
where we have used the condition $n_{2} \geq  \frac{2\log{p}}{\log\left(1+\frac{a^{2}}{8\sigma^{2}}\right)}+1$.
The expression in \eqref{rhss} tends to 0 as $p\to \infty$. This completes the proof.
\end{proof}

Proposition \ref{prop:ad} shows that the condition  
$$ n = \Omega \left( s\log{\left(\frac{ep}{s}\right)} \vee \frac{\log{p}}{\log\left(1+\frac{a^{2}}{8\sigma^{2}}\right)}\right) $$
is sufficient for exact recovery without knowing the sparsity parameter $s$. 

We now turn to the case where both $s$ and $a$ are unknown, but $\sigma$ is known. In Proposition \ref{prop:ad}, we have used the condition
$$ n_{2} >  \frac{2\log{p}}{\log\Big(1+\frac{a^{2}}{8\sigma^{2}}\Big)},$$
which is equivalent to
\begin{equation}\label{eq:a2} 
a^{2} > 8\sigma^{2} \left(p^{\frac{2}{n_{2}}}-1\right).
\end{equation} 
This inspires us to replace the threshold function $t(u)=a\|u\|/2$ considered in Proposition \ref{prop:ad} by
\begin{equation}\label{eq:thres3}
    t(u)=\sigma\sqrt{2\left(p^{\frac{2}{n_{2}}}-1\right)}\|u\|,\quad u \in \mathbb{R}^{n_2}.
\end{equation}
Then, we get the following result analogous to Proposition \ref{prop:ad}.
\begin{prop}\label{th:ad}
{Let the assumptions of Proposition \ref{prop:ad} be satisfied.} 
Let the threshold $t(\cdot)$ be defined in \eqref{eq:thres3}. Then, the selector $\hat{\eta}$ defined in  \eqref{eq:selec1} - \eqref{eq:selec2} achieves exact recovery under both risks (Hamming and support recovery) if $n \geq 2\bigg(C_{0}s\log\big(\frac{ep}{s}\big) \vee \frac{2\log{p}}{\log\left(1+\frac{a^{2}}{8\sigma^{2}}\right)} \bigg).$
\end{prop}
\begin{proof}
Acting as in the proof of Theorem \ref{thm:suf1} and choosing there $\delta = 1$ we get
$$  \underset{\beta \in \Omega^{p}_{s,a}}{\sup}\mathbf{E}_{\beta}\left| \hat{\eta} - \eta_{\beta}\right| \leq p\mathbf{P}\left( \sqrt{2}\sigma \left|\varepsilon \right|>t\left(\zeta\right) \right) + s \mathbf{P}\left( \sqrt{2}\sigma \left|\varepsilon \right|\ge\left(a\|\zeta\|-t\left(\zeta \right)\right)_{+} \right) +C_1p\left(\frac{s}{2p}\right)^{C_2s} 
$$
where $\varepsilon$ is a standard Gaussian random variable and $\zeta \sim \mathcal{N}\left(0,\mathbb{I}_{n_2}\right)$ is a standard Gaussian random vector in $\mathbb{R}^{n_2}$ independent of $\varepsilon$.
Since $n_2 \geq  \frac{2\log{p}}{\log\left(1+\frac{a^{2}}{8\sigma^{2}}\right)}$, we have \eqref{eq:a2}, which implies $a\|\zeta\| \geq 2t\left(\zeta\right)$. Therefore, 
 \begin{equation}\label{eq:th:ad} \underset{\beta \in \Omega^{p}_{s,a}}{\sup}\mathbf{E}_{\beta}\left| \hat{\eta} - \eta_{\beta}\right| \leq 2p\mathbf{P}\left( \sqrt{2}\sigma \left|\varepsilon \right|\ge t\left(\zeta\right) \right)  +C_1p\left(\frac{s}{2p}\right)^{C_2s}.
 \end{equation}
The second summand on the right hand side of \eqref{eq:th:ad} is treated in the same way as in 
Proposition \ref{prop:ad}. 
To bound  the first summand, we note that due to \eqref{eq:thres3},
$$
\mathbf{P}\left( \sqrt{2}\sigma \left|\varepsilon \right|\ge t\left(\zeta\right) \right)= 
\mathbf{P}\left( |Z|\ge\sqrt{n_2\big(p^{\frac{2}{n_{2}}}-1\big)} \right)
$$
where $Z$ is a Student random variable with $n_2$ degrees of freedom.  Using the inequalities $n_2\big(p^{\frac{2}{n_{2}}}-1\big)= n_2\big(\exp(2(\log p)/n_{2})-1\big)\ge 2\log p$, $n_2\ge C_0 \log p$, and Lemma \ref{lem:ck}
we find 
$$
\mathbf{P}\left( |Z|\ge\sqrt{n_2\big(p^{\frac{2}{n_{2}}}-1\big)} \right)\le \frac{p^{-1+1/n_2}}{\sqrt{2\log p}}\le \frac{p^{-1}\exp(1/C_0)}{\sqrt{2\log p}}.
$$
This implies that the first summand on the right hand side of \eqref{eq:th:ad} tends to 0 as $p\to \infty$.
\end{proof}

Thus, if only $\sigma$ is known while $a$ and $s$ are not, we can achieve exact recovery under the same condition as for the ML decoder (which is not computationally tractable and depends on $s$). Next, we show that, replacing $\sigma$ in \eqref{eq:thres3} by a suitable estimator, we can render the procedure completely adaptive to all parameters of the problem. 

Define $\hat \sigma>0$ by 
$$ 
\hat\sigma^2 = \frac{1}{n_{2}} \sum_{i \in \mathcal{D}_{2}}\Big(Y_{i}-\sum_{j=1}^{p}X_{ij}\hat{\beta}_{j}\Big)^{2},
$$
where $\hat{\beta}$ is the same Square-Root SLOPE estimator as in \eqref{eq:selec2} 
%
and consider the threshold function
\begin{equation}\label{eq:thres4}
    t(u)=\hat{\sigma}\sqrt{2\left(p^{\frac{2}{n_{2}}}-1\right)}\|u\|,\quad \forall u \in \mathbb{R}^{n_2}.
\end{equation}
We get the following result for the fully adaptive procedure corresponding to this threshold.
\begin{prop}
{Let the assumptions of Proposition \ref{prop:ad} be satisfied.} 
Let the threshold $t(\cdot)$ be defined in \eqref{eq:thres4}. Then, there exists a constant $\bar C_0>0$ such that the selector $\hat{\eta}$ defined in  \eqref{eq:selec1} - \eqref{eq:selec2} achieves exact recovery under both risks (Hamming and support recovery) if $n \geq 2\bigg({\bar C}_{0}s\log\big(\frac{ep}{s}\big) \vee \frac{2\log{p}}{\log\left(1+\frac{a^{2}}{16\sigma^{2}}\right)} \bigg).$
    \end{prop}
\begin{proof}
Define the random event
$$ \mathbb{B} = \left\{ \|\hat{\beta}-\beta \|^{2} \leq \sigma^{2}\right\}\cap\left\{ \left|\frac{\hat{\sigma}^{2}}{\|\hat{\beta}-\beta\|^{2}+\sigma^{2}}-1\right|\leq \frac{1}{2} \right\}. $$
We have
$$\underset{\beta \in \Omega_{s,a}}{\sup}\mathbf{E}_{\beta}\left| \hat{\eta} - \eta_{\beta} \right| \leq \underset{\beta \in \Omega_{s,a}}{\sup}\mathbf{E}_{\beta}\big( \left| \hat{\eta} - \eta_{\beta} \right|\mathbf{1}\{ \mathbb{B}\} \big) + p\underset{\beta \in \Omega_{s,a}}{\sup}\mathbf{P}_{\beta}\left( \mathbb{B}^{c}\right).$$
To control the second term on the right hand side, note that, 
conditionally on $\hat{\beta}$, the estimator $ \hat{\sigma}^{2}$ has the same distribution as
$$ \frac{\|\hat{\beta}-\beta\|^{2}+\sigma^{2}}{n_{2}} \chi^{2}(n_2),$$
where  $\chi^{2}(n_{2})$ is a chi-squared random variable with $n_{2}$ degrees of freedom. We will use the following lemma, cf. \cite{cgpt} or  \cite{lptv}.
\begin{lem}\label{lem:chi2} For any $N\ge 1$ and $t>0$, 
$$
\mathbf{P}(\left|\chi^{2}(N)/N-1\right|\ge t) \le 2\exp\left(-\frac{t^2N}{4(1+t)}\right),
$$
where $\chi^{2}(N)$ is a chi-squared random variable with $N$ degrees of freedom. 
\end{lem}
From Lemma \ref{lem:chi2} with $t=1/2$ and Proposition \ref{prop:1} we get 
$$  p\underset{\beta \in \Omega_{s,a}}{\sup}\mathbf{P}_{\beta}\left( \mathbb{B}^{c}\right) \leq C_1p\left(\frac{s}{2p}\right)^{C_2s}+2p e^{-(n_2-1)/24}. $$
Here, $p\big(\frac{s}{2p}\big)^{C_2s} \to 0$ as $p\to \infty$ (cf. the proof of 
Proposition \ref{prop:ad}), while  $p e^{-n_2/24}  \to 0$ as $p\to \infty$ provided that we choose $\bar C_0> 24$. 

To evaluate $\Gamma:=\mathbf{E}_{\beta}\big( \left| \hat{\eta} - \eta_{\beta} \right|\mathbf{1}\{ \mathbb{B}\} \big)$, we act similarly to the proof of Theorem \ref{thm:suf1}. 
We have 
\begin{eqnarray*} 
\Gamma 
&= &
\sum_{i:\beta_{i}=0}\mathbf{P}_{\beta}(\{|\alpha_{i}|>t(X_i)\} \cap  \mathbb{B} ) +
\sum_{i:\beta_{i}\neq 0}\mathbf{P}_{\beta}(\{|\alpha_{i}|\le t(X_i)\} \cap  \mathbb{B}). 
\end{eqnarray*}
Set 
$\sigma_{*}=\sqrt{\|\hat{\beta}-\beta\|^{2}+\sigma^{2}} 
.$
On the event $\mathbb{B}$, we have $ \sigma^{2}_*\leq 2\hat{\sigma}^{2} \leq 3 \sigma^{2}_*$ and $\sigma^{2}_*\leq 2\sigma^{2}$. The last inequality and the assumption on $n_{2}$ imply that $a\ge 2\sqrt{2}\sigma^{*}(p^{{2}/{n_2}}-1)^{1/2} $. Using these remarks and the fact that, conditionally on $\hat{\beta}$ and $X_i$, the variable $\alpha_{i}$ has the same distribution as \eqref{alphai} we obtain, for all $i$ such that $\beta_{i}=0$, 
\begin{eqnarray*}
\mathbf{P}_{\beta}(\{|\alpha_{i}|>t(X_i)\} \cap  \mathbb{B} ) 
&\le& \mathbf{P}_{\beta}(\{|\alpha_{i}|>\sigma_* \|X_{i}\|(p^{{2}/{n_2}}-1)^{1/2} \} \cap  \mathbb{A} ) 
\\
&\le& \mathbf{P}_{\beta}(|\varepsilon|>\|X_{i}\|(p^{{2}/{n_2}}-1)^{1/2} \}),
\end{eqnarray*}
where $ \mathbb{A} = \big\{ \|\hat{\beta}-\beta \|^{2} \leq \sigma^{2}\big\}$ and $\varepsilon$ is a standard Gaussian random variable independent of $\|X_{i}\|$. Similarly, for all $i$ such that $\beta_{i}\ne 0$ (and thus 
$|\beta_{i}|\ge a$) we have
\begin{eqnarray*}
\mathbf{P}_{\beta}(\{|\alpha_{i}|\le t(X_i)\} \cap  \mathbb{B} ) 
&\le& \mathbf{P}_{\beta}(\{|\alpha_{i}|\le \sqrt{3}\sigma_* \|X_{i}\|(p^{{2}/{n_2}}-1)^{1/2} \} \cap  \mathbb{A} ) 
\\
&\le& \mathbf{P}_{\beta}(\sigma_*|\varepsilon|\ge a\|X_{i}\|-\sqrt{3}\sigma_*\|X_{i}\|(p^{{2}/{n_2}}-1)^{1/2})
\\
&\le& \mathbf{P}_{\beta}(|\varepsilon|\ge (2\sqrt{2}-\sqrt{3})\|X_{i}\|(p^{{2}/{n_2}}-1)^{1/2}).
\end{eqnarray*}
Combining the above inequalities we find
\begin{eqnarray} \label{gg}
\Gamma \leq p\mathbf{P}\left(|Z| \geq (p^{{2}/{n_2}}-1)^{1/2} \right) + s \mathbf{P}\left(|Z| \geq (2\sqrt{2}-\sqrt{3})(p^{{2}/{n_2}}-1)^{1/2}  \right),
\end{eqnarray}
where $Z$ is a Student random variable with $n_{2}$ degrees of freedom. Finally, we apply the same argument as in the proof of Proposition \ref{th:ad} to obtain that the right hand side of \eqref{gg} vanishes as $p\to\infty$.
\end{proof}

\section{Generalization to sub-Gaussian distributions}\label{sec:sub}

In this section, we generalize our procedure to the case where both the design (sensing) matrix $X$ and the noise $\xi$ are sub-Gaussian. 
Recall that, for given $\sigma_{\zeta}>0$, a random variable $\zeta$ is called $\sigma_{\zeta}$-sub-Gaussian if
$$
\E \exp(t \zeta) \le \exp(\sigma_{\zeta}^2t^2/2), \quad \forall t\in \mathbb{R}. 
$$
In particular, this implies that $\zeta$ is centered. 

In this section, we assume that both $X$ and $\xi$ have i.i.d. sub-Gaussian entries, and as above, $X$ is independent of $\xi$. 

The estimation part of our procedure (cf. Proposition \ref{prop:1}) extends to sub-Gaussian designs as follows. 

{
\begin{prop}\label{prop:subGaussian}
Assume that the entries of matrix $X$ are i.i.d. $\sigma_X$-sub-Gaussian random variables, the entries of the noise $\xi$ are i.i.d. $\sigma$-sub-Gaussian random variables for some $\sigma >0$, $\mathbf{E}(X_{ij}^2)=1$ for all entries $X_{ij}$ of matrix $X$, and $X$ is independent of $\xi$.
Let $\hat\beta$ be the Square-Root SLOPE estimator  defined in Section \ref{sec:contr} with large enough $A>0$ depending only on~$\sigma, \sigma_X$. There exist constants $C_{0},C_{1},C_{2}>0$ that can depend only on $\sigma_X$, such that for all $\delta \in (0,1]$ and $n_1 > \frac{C_{0}}{\delta^2}s\log{\left(\frac{ep}{s}\right)}$ we have
$$ \underset{|\beta|_{0}\leq s}{\sup} \mathbf{P}_{\beta}\left( \|\hat{\beta} - \beta\| \geq \delta\sigma  \right) \leq C_1 \Big(\frac{s}{2p}\Big)^{C_{2}s}.$$
\end{prop}
}
The proof of this proposition is based on combination of arguments from \cite{bellec} and \cite{collier2018}. 
It is given in the Appendix.

We will also need the following lemma proved   in the Appendix.
\begin{lem}\label{lem:3}
Let $U,V$ be two independent random vectors in $\mathbb{R}^{n}$, such that the entries of $U$ are i.i.d. random variables and the entries of $V$ are i.i.d. $\sigma$-sub-Gaussian random variables for some $\sigma$. Assume that $\mathbf{E}(U_i^2)=1$ and $ \mathbf{E}(U_i^4)\le \sigma_1^4$ for all components $U_i$ of $U$, where $\sigma_1>0$. Then, for any $t>0$,
$$ \mathbf{P}\left( \frac{\left|U^{\top}V\right|}{\|U\|^{2}} \geq t\right) \leq 2\exp\left(-\frac{nt^2}{8\sigma^2}\right) + \exp\left(-\frac{9n}{32\sigma_1^4}\right).$$
\end{lem}

We are now ready to state a general result for sub-Gaussian designs.

\begin{theorem}\label{thm:subgaussian}
Let the assumptions of Proposition \ref{prop:subGaussian} be satisfied.
Let $n\ge 4$ be an even integer and $2(1\vee 1/C_2)\le s<p$. Set $n_{1} = n_{2} = n/2$, and let the threshold $t(\cdot)$ be defined in \eqref{eq:thres2}. Then, there exists a constant $C>0$ such that the selector $\hat{\eta}$ defined in  \eqref{eq:selec1} - \eqref{eq:selec2} achieves exact recovery under both risks (Hamming and support recovery) if $n \geq C \left( s\log\left(\frac{ep}{s}\right) \vee  \frac{\sigma^{2}\log{p}}{a^{2}} \right)$.
    \end{theorem}
\begin{proof}
We act similarly to the proof of Theorem \ref{thm:suf1} where we set $\delta=1$ and $t(X_{i}) = \frac{a}{2}\|X_{i}\| $. 
Then, for all $i$ such that $\beta_i=0$, we have
$$
\mathbf{P}_{\beta}(|\alpha_{i}|>t(X_i)| {\mathcal D}_1)\le \mathbf{P}\bigg(\frac{\left|U^{\top}V\right|}{\|U\|^2} >\frac{a}{2}  \Big| \, {\mathcal D}_1\bigg),
$$
where $U=X_{i}$ and $V=Y - \sum_{j\neq i}X_{j}\hat{\beta}_{j}=\sigma\xi+ \sum_{j\neq i}X_{j}({\beta}_{j} -\hat{\beta}_{j})$. For fixed $\hat{\beta}$, the components of $V$ are i.i.d. $\sigma_*$-sub-Gaussian with $\sigma_*= (\sigma^{2}+ \|\hat{\beta}-\beta\|^{2})^{1/2}$. In particular, for fixed $\hat{\beta}$ on the event  $ \mathbb{A} = \{ \|\hat{\beta}-\beta \|^{2} \leq \sigma^{2}  \},$ they are $\sqrt{2}\sigma$-sub-Gaussian.
Thus, from Lemma \ref{lem:3} we obtain that there exists an absolute constant $c>0$ such that, for all $i$ with $\beta_i=0$,
$$ \mathbf{P}_{\beta}(\{ |\alpha_{i}|>t(X_i)\}  \cap \mathbb{A} ) \le 
2\exp\Big(-c n_{2}\Big(\frac{a^{2}}{\sigma^{2}}  \wedge 1\Big)\Big).$$
The same bound holds for $ \mathbf{P}_{\beta}(\{ |\alpha_{i}|\le t(X_i)\}  \cap \mathbb{A} ) $ for all $i$ such that $\beta_i\ne 0$.  The rest of the proof follows the same lines as the proof of Theorem \ref{thm:suf1} using Proposition \ref{prop:subGaussian} to evaluate $ \mathbf{P}_{\beta}(\mathbb{A}^c)$. This yields the bound
$$  \underset{\beta \in \Omega^{p}_{s,a}}{\sup}\mathbf{E}_{\beta}\left| \hat{\eta} - \eta_{\beta} \right| \leq  4p\exp\Big(-c n_{2}\Big(\frac{a^{2}}{\sigma^{2}}  \wedge 1\Big)\Big) + C_1 p\Big(\frac{s}{2p}\Big)^{C_{2}s}.$$
The second summand on the right hand side of this inequality vanishes as $p\to \infty$ as shown in the proof of
Proposition \ref{prop:ad}. 
The second summand vanishes as $p\to \infty$ if $n_2 > c'\big(\frac{\sigma^{2}}{a^{2}}\vee 1\big) \log p$ for some $c' > 1/c$. We conclude the proof by noticing that $s\log\left(\frac{ep}{s}\right)\ge \log p$ for all $1\le s\le p$. 
\end{proof}
Theorem \ref{thm:subgaussian} shows that, with no restriction on the joint behavior of $s$, $a$ and $\sigma$, a sufficient condition for exact recovery in the sub-Gaussian case is the same as in the Gaussian case:
$$ 
n = \Omega \left( s\log{\left(\frac{ep}{s}\right)} \vee \frac{\sigma^{2}\log p}{a^{2}} \right). 
$$
On the other hand, necessary conditions of exact recovery given in \eqref{wain} are valid  for any $X$ with i.i.d. centered entries satisfying $\mathbf{E}(X_{ij}^2)=1$ and for Gaussian noise $\xi$ \cite{wangWain}. It follows 
that, if under the assumptions of Theorem \ref{thm:subgaussian} the noise $\xi$ is Gaussian, our selector achieves {the exact phase transition in the zone $a/\sigma = \mathcal{O}(1/\sqrt{s})$,} while for other values of $s$, $a$ and $\sigma$, it achieves the phase transition up to a logarithmic factor.

\section{ Robustness through MOM thresholding}\label{sec:rob}

In the previous section, we have shown that the suggested selector succeeds for independent sub-Gaussian designs. In practice, the observations we have may be corrupted by some outliers, and the assumption of sub-Gaussian noise is not always relevant. This motivates us to introduce a robust version of this selector.
In this section, we propose a selector that achieves similar properties as described above  under weaker assumptions on the noise and in the presence of outliers.

Suppose that data are partitioned in two disjoint groups $O$ and $I$, where $({\bf x}_{i},Y_{i})_{i \in O}$ are outliers, that is arbitrary vectors with ${\bf x}_i\in \mathbb{R}^p$,  $Y_i\in \mathbb{R}$, and $({\bf x}_{i},Y_{i})_{i\in I}$ are informative observations distributed as described below. Here, $|I|+|O|=n$. 

We assume that the informative observations satisfy
\begin{equation}\label{model2}
    Y_i = {\bf x}_i^{\top}\beta + \xi_i, \quad i\in I, 
\end{equation}
where $\beta\in \mathbb{R}^p$ is an unknown vector of parameters and $\xi_1,\dots,\xi_n$ are zero-mean i.i.d. random  variables such that for some $q,\sigma>0$ we have $\mathbf{E}(|\xi_{i}|^{2+q}) \leq \sigma^{2+q},  i \in I$.  We also assume that, { for $i \in I$}, all components $X_{ij}$ of vectors ${\bf x}_{i}$ are $\sigma_X$-sub-Gaussian i.i.d. random  variables with zero mean and~$\mathbf{E}(X_{ij}^{2}) = 1$. Here, $\sigma_X>0$ is a constant. The conditions on the design can be further weakened but we consider sub-Gaussian designs for the sake of readability and also because such designs are of major interest in the context of compressed sensing. We also assume that $\xi = (\xi_1,\dots,\xi_n)$ is independent of $X=({\bf x}_1^\top,\dots,{\bf x}_n^\top)^\top$. 

In this section, we propose a selector based on median of means (MOM). The idea of MOM goes back to \cite{MOM3},  \cite{MOM2}, \cite{MOM1}. Our selector uses again sample splitting. We first construct a preliminary estimator $\hat{\beta}^*$ based on the subsample $\mathcal{D}_1$ and then we threshold debiased estimators of the components of $\beta$. These debiased estimators are constructed using both $\hat{\beta}^*$ and the second subsample $\mathcal{D}_2$. {
In the same spirit as in Proposition \ref{prop:1}, we require $\hat{\beta}^{*}$ to satisfy the following assumption. 
{
\begin{assumption}\label{MOM-SLOPE}
Let $X$ and $\xi$ satisfy the conditions stated above in this section. 
There exist  constants $c_0, c_{1},c_{2}>0$ depending only on $q$ and the sub-Gaussian constant $\sigma_X$ such that the following holds. If $|O| \leq c_{0}s\log(ep/s) \leq n_1/2$, then the estimator $\hat{\beta}^*$ satisfies
$$  \underset{|\beta|_0 \le s}{\sup}\mathbf{P}_{\beta}\left( \| \hat{\beta}^* - \beta \| \geq  \sigma\right) \leq c_1 \Big(\frac{s}{p}\Big)^{c_{2}s} .$$
\end{assumption} 
}
As a preliminary estimator, we may take the MOM-SLOPE estimator of \cite{lec17}, for 
which Assumption~\ref{MOM-SLOPE} is satisfied, cf. Lemma~\ref{lem:MOM_SLOPE}.
}

Note that the bound of Assumption \ref{MOM-SLOPE} holds uniformly over all outlier sets $|O|$ such that $|O| \leq c_{1}s\log(ep/s)$, and uniformly over all distributions of $\xi_i$ satisfying the assumptions of this section.  Based on the fact that the MOM-SLOPE estimator satisfies Assumption \ref{MOM-SLOPE}, we will now present a robust version of our selector. We split our sample in two subsamples of size $n/2$ each. The first subsample is used to construct a pilot estimator, which is the MOM-SLOPE estimator or any other estimator $ \hat{\beta}^*$ satisfying Assumption \ref{MOM-SLOPE}. Then, the selector is constructed based on this estimator $\hat{\beta}^*$ and on the second subsample.  To simplify the notation, for the rest of this section we will consider that the size of the second subsample is $n$ rather than $n/2$ and we have an estimator $\hat{\beta}^*$ satisfying Assumption \ref{MOM-SLOPE} and independent from the second subsample. 

Let $K =\lfloor c_3\log(p) \rfloor$ be the number of blocks, {with  $c_3\geq 500$}. Assume that $1< K<n$. By extracting $K$ disjoint blocks from the observation $Y$ corresponding to the second subsample, we get $K$ independent observations $({\bf Y}^{(i)})_{1\leq i \leq K}$, where ${\bf Y}^{(i)} \in \mathbb{R}^{q}$ and $q = \lfloor \frac{n}{K} \rfloor $. Each observation  ${\bf Y}^{(i)}$ satisfies
$$ {\bf Y}^{(i)} = {\bf X}^{(i)}\beta + \xi^{(i)},$$
where $X^{(i)}$ is a submatrix of $X$ with rows indexed by the $i$th block.
For $i=1,\dots,K$, consider the new observations
\begin{align*}
     {\bf Z}^{(i)} &= \frac{1}{q} {\bf X}^{(i)\top}{\bf Y}^{(i)} - \left( \frac{1}{q}{\bf X}^{(i)\top}{\bf X}^{(i)} - I_{p} \right)\hat{\beta}^*.
\end{align*}
We denote by $Z^{(i)} _1,\dots,Z^{(i)} _p$ the components of ${\bf Z}^{(i)}$.
Consider the selector defined as a vector
\begin{equation}\label{eq:selec4}
    \hat{\eta}(X,Y) = \left( \hat{\eta}_{1}(X,Y),\dots,\hat{\eta}_{p}(X,Y) \right) 
\end{equation}
with components 
\begin{equation}\label{eq:selec5}
    \hat{\eta}_{j}(X,Y) = \mathbf{1}\left\{  |Med(Z_{j})| > t \right\},\quad j=1,\dots,p,
\end{equation}
where $Med(Z_{j})$ is the median of $Z^{(1)} _j,\dots,Z^{(K)} _j$, and $t=c_4\sigma \sqrt{\frac{\log{p}}{n}}$ with {a positive constant $c_4>0$  depending only on the sub-Gaussian constant $\sigma_X$}. The next theorem shows that, when the noise has polynomial tails and contains a portion of outliers, the robust selector \eqref{eq:selec4} - \eqref{eq:selec5} achieves exact recovery under the same condition on the sample size as when the noise is Gaussian. 
\begin{theorem}\label{thm:robust}
Let $X$ and $\xi$ satisfy the conditions stated at the beginning of this section. Then, there exist absolute constants $c', c_{3},c_{4}>0$ and a constant $C'>0$ depending only on $q$ and on the sub-Gaussian constant $\sigma_X$ such that the following holds.
 Let $c'< s<p$. Then, the selector given in \eqref{eq:selec4} - \eqref{eq:selec5} achieves exact recovery with respect to both risks \eqref{eq:exact:supp} and \eqref{eq:exact:hamm}
if $n\geq C' \left( s\log(p/s) \vee \sigma^{2}\frac{\log(p)}{a^{2}} \right)$ and $|O| < K/4$. 
\end{theorem}

\begin{proof}

For all $i=1,\dots,K$, we have $$  {\bf Z}^{(i)} = \beta + \varepsilon^{(i)}, $$
where 
$$ \varepsilon^{(i)} = \left( \frac{1}{q}{\bf X}^{(i)\top}{\bf X}^{(i)} - I_{p}\right)\left( \beta - \hat{\beta}^*\right) + \frac{1}{q} {\bf X}^{(i)\top}\xi^{(i)}. $$
The random vectors $\varepsilon^{(1)}, \dots, \varepsilon^{(K)}$ are independent conditionally on $\hat{\beta}^*$. Let $\varepsilon^{(i)}_j$ denote the $j$th component of $\varepsilon^{(i)}$. Note that $Med({Z}_{j}) = \beta_{j} +  Med(\varepsilon_{j}) $, where $Med(\varepsilon_{j})$ denotes the median of $\varepsilon^{(1)}_{j},\dots,\varepsilon^{(K)}_{j}$. 
Choose $C'>0$ large enough to guarantee that $a > 2t$. Then,
\begin{align*}
  \mathbf{E}_{\beta}\left| \hat{\eta} - \eta_{\beta} \right| &= \sum_{j:\beta_{j} \neq 0} \mathbf{P}_{\beta}\left( |Med(Z_{j})| \leq t \right) + \sum_{j:\beta_{j} = 0} \mathbf{P}_{\beta}\left( |Med(Z_{j})| > t \right) \\
   &\leq \sum_{j:\beta_{j} \neq 0} \mathbf{P}_{\beta}\left( |Med(\varepsilon_{j})| \geq a-t\right) + \sum_{j:\beta_{j} = 0}\mathbf{P}_{\beta}\left( |Med(\varepsilon_{j})| > t \right)\\
   &\leq p\, \underset{j=1,\dots,p}{\sup}\,\mathbf{P}_{\beta}\left( |Med(\varepsilon_{j})| \geq t \right) .
\end{align*}
Consider the event
  $ \mathbb{A}_* = \{ \|\hat{\beta}^*-\beta \|^{2} \leq \sigma^{2}  \}. $ The following lemma is proved in the Appendix.
\begin{lem}\label{lem:7}
    Under the conditions of Theorem \ref{thm:robust} we have
    $$\underset{j=1,\dots,p}{\sup}\mathbf{P}_{\beta}\left( |Med(\varepsilon_{j})| \geq t \right)  \leq e^{- {K/400}} + \mathbf{P}_{\beta}(\mathbb{A}_*^{c})$$
    for some $c_5>0$.
\end{lem}
From Lemma \ref{lem:7} and Assumption \ref{MOM-SLOPE} we get
$$ \underset{\beta \in \Omega^p_{s,a}}{\sup}\mathbf{E}_{\beta}\left|\hat{\eta} - \eta_{\beta}\right| \leq p e^{- {K/400}} +  pe^{-c_{2}s\log(ep/s)}. $$
Since $K = \lfloor c_3\log(p)\rfloor$, and $s\log(ep/s)\ge c'\log(ep/c')$ the result follows for {$c_3\geq 500$,} and $c'>0$ chosen large enough.
\end{proof}
We see that sufficient conditions of exact recovery for the robust selector are of the same order as in the Gaussian case. If the risk is considered uniformly over all noise distributions under the conditions of this section, clearly the Gaussian noise is in this class. Hence, necessary conditions in the Gaussian case are also necessary for such a uniform risk over noise distributions. We have proved previously that, sufficient conditions for the selector \eqref{eq:selec1} - \eqref{eq:selec2} to achieve exact recovery are almost optimal in the Gaussian case. As a consequence, the selector \eqref{eq:selec4} - \eqref{eq:selec5} is almost optimal in this more general setting. 

\section{Conclusion}
In this paper, we proposed computationally tractable algorithms of variable selection that can achieve exact recovery under milder conditions than the ones known so far. Throughout different sections, we have investigated, respectively, the setting with Gaussian observations, sub-Gaussian observations, and heavy-tailed observations corrupted by arbitrary outliers. We have shown that the suggested selectors nearly achieve necessary conditions of exact recovery.  For the Gaussian case, we obtained not only the conditions of exact recovery but also accurate upper and lower bounds on the minimax Hamming risk.  Furthermore, we constructed a selector, which is fully adaptive to all parameters of the problem and achieves exact recovery under almost the same sufficient conditions as in the case where sparsity $s$ and the signal strength $a$ and the noise level $\sigma$ are known. Finally, we proposed a robust variant of our method that achieves exact recovery when the observations have outliers or are heavy-tailed  under sufficient conditions similar to those for the Gaussian case.

\medskip

{\bf Acknowledgement.} This work was supported by GENES and by the French National Research Agency (ANR) under the grants IPANEMA (ANR-13-BSH1-0004-02) and Labex Ecodec (ANR-11-LABEX-0047).

%

\appendix
\section{Appendix}
In order to prove  Theorem \ref{th:inf}, we use the following result from \cite{butucea}.  Consider the set of binary vectors
$$ A = \left\{  \eta \in \{0,1\}^{p} : \  |\eta|_{0} \leq s \right\}$$
and assume that we are given a family $\{\mathbf{P}_{\eta}, \eta \in A\}$ where each $\mathbf{P}_{\eta}$ is a probability distribution on a measurable space $(\mathcal{X}, \mathcal{U})$.
We observe $X$  drawn from $\mathbf{P}_{\eta}$ with some unknown $\eta=(\eta_1,\dots,\eta_p) \in A$ and we consider the Hamming risk of a selector $\hat{\eta}=\hat{\eta}(X)$:
$$
\sup_{ \eta \in A}\mathbf{E}_{\eta}|\hat{\eta} - \eta|
$$
where $\mathbf{E}_{\eta}$ is the expectation w.r.t. $\mathbf{P}_{\eta}$. We call the selector any estimator with values in $\{0,1\}^{p}$. Let $\pi$ be a probability measure on $\{0,1\}^{p}$ (a prior on $\eta$). We denote by $\mathbb{E}_{\pi}$ the expectation with respect to~$\pi$. Then the following result is proved in \cite{butucea}
\begin{theorem}{\cite{butucea}}\label{thm:indep_prior} 
Let $\pi$ be a product on $p$ Bernoulli measures with parameter $s'/p$ where $s'\in (0, s]$. Then,
\begin{equation}\label{eeq2:th}
\inf_{\hat{\eta}} \sup_{ \eta \in A} \mathbf{E}_\eta |\hat{\eta} - \eta|
\geq  \inf_{\hat T\in [0,1]^p}\mathbb{E}_{\pi}\mathbf{E}_\eta \sum_{i=1}^p | {\hat T}_i - \eta_i | - 4 s' \exp\Big(-\frac{(s-s')^2}{2s}\Big),
\end{equation}
where $\inf_{\hat{\eta}}$ is the infimum over all selectors and $\inf_{\hat T\in [0,1]^p}$ is the infimum over all  estimators $\hat T=({\hat T}_1,\dots,{\hat T}_p)$ with values in $[0,1]^p$. 
\end{theorem}
\begin{proof}[\textbf{Proof of Theorem \ref{th:inf}}]
Let $\Theta(p,s,a)$ a subset of $\Omega_{s,a}^{p}$ defined as
$$ \Theta(p,s,a) = \{ \beta \in \Omega_{s,a}^{p}: \ \beta_{i} = a , \ \forall i \in S_{\beta} \}.$$
Since any $\beta \in \Theta(p,s,a)$ can be written as $\beta = a \eta_{\beta}$,  there is a one-to-one correspondence between $A$ and $\Theta(p,s,a)$. Hence,
$$  \inf_{\hat{\eta}} \sup_{ \eta \in A} \mathbf{E}_\eta |\hat{\eta} - \eta| = \inf_{\hat{\eta}} \sup_{ \beta \in \Theta(p,s,a)} \mathbf{E}_\beta |\hat{\eta} - \eta_{\beta}| .$$
Using this remark and Theorem \ref{thm:indep_prior} we obtain that, for all $s'\in (0, s]$,
$$ \underset{\hat{\eta}}{\inf}\underset{\beta \in \Omega^{p}_{s,a}}{\sup}\mathbf{E}_{\beta}|\hat{\eta} - \eta_{\beta}| \geq  \inf_{\hat T\in [0,1]^p}\mathbb{E}_{\pi}\mathbf{E}_\eta \sum_{i=1}^p | {\hat T}_i(X,Y) - \eta_i | - 4 s' \exp\Big(-\frac{(s-s')^2}{2s}\Big),$$
where $\pi$ a product on $p$ Bernoulli measures with parameter $s'/p$.
Thus, to finish the proof it remains to show that
$$ \inf_{\hat T\in [0,1]^p}\mathbb{E}_{\pi}\mathbf{E}_\eta \sum_{i=1}^p | {\hat T}_i(X,Y) - \eta_i | \geq \frac{s'}{s}\psi_{+}(n,p,s,a,\sigma). $$
We  first  notice that
\begin{eqnarray*} \inf_{\hat T\in [0,1]^d}
\mathbb{E}_{\pi}\mathbf{E}_\eta \sum_{i=1}^p | {\hat T}_i(X,Y) - \eta_i | &\geq& \sum_{i=1}^{p}\mathbb{E}_{\pi}\mathbf{E}_\eta\Big[\inf_{\hat T_{i}\in [0,1]}\mathbb{E}_{\pi}\mathbf{E}_\eta \Big( | {\hat T}_i(X,Y) - \eta_i | \big| \eta_{(-i)}, X\Big)\Big]\\
&\ge &\sum_{i=1}^{p}\mathbb{E}_{\pi}\mathbf{E}_\eta\Big[\inf_{\tilde T_{i}\in [0,1]}\mathbb{E}_{\pi}\mathbf{E}_\eta \Big( | {\tilde T}_i(X,Y,\eta_{(-i)}) - \eta_i | \big| \eta_{(-i)}, X\Big)\Big]\\
 &=&\sum_{i=1}^{p}\mathbb{E}_{\pi}\mathbf{E}_\eta [ L_i^*]
\end{eqnarray*}
where $\eta_{(-i)}$ denotes $(\eta_{j})_{j \neq i}$ and $L_i^*=L_i^*(\eta_{(-i)}, X)$ has the form
\begin{eqnarray}\nonumber
L_i^*&=& \inf_{\tilde T_{i}\in [0,1]}\Big( \frac{s'}{p}\int ( 1 - {\tilde T}_i(X,y,\eta_{(-i)}) )\varphi_\sigma(y-aX_i-\sum_{j \neq i} a\eta_{j}X_{j}){\rm d}y \\
&& \quad + \Big(1 - \frac{s'}{p} \Big) \int {\tilde T}_i(X,y,\eta_{(-i)}) \varphi_\sigma(y-\sum_{j \neq i} a\eta_{j}X_{j}){\rm d}y \Big). \label{li}
\end{eqnarray}
Here, $\varphi_{\sigma}$ is the density of Gaussian distribution in ${\mathbb R}^n$ with i.i.d. zero-mean and variance $\sigma^{2}$ components.
By the Bayesian version of the Neyman-Pearson lemma, the infimum in \eqref{li} is attained for $\tilde{T}_{i} = T_{i}^{*}$ given by the formula
$$ T_{i}^{*}\left(X, Y, \eta_{(-i)}\right) = \mathbf{1}\left\{ \frac{(s'/p)\phi_{\sigma}(Y-aX_i-\sum_{j \neq i} a\eta_{j}X_{j})}{(1-s'/p)\phi_{\sigma}(Y-\sum_{j \neq i} a\eta_{j}X_{j})}>1\right\}. $$
Equivalently,
$$  T_{i}^{*} = \mathbf{1}\left\{ \frac{X_{i}^{\top}(Y-\sum_{j \neq i} a\eta_{j}X_{j})}{\| X_{i}\|}>t(s',X_{i})\right\}, $$
where 
$$ t(s',X_{i})= \frac{a\|X_{i}\|}{2} + \frac{\sigma^{2}\log(\frac{p}{s'}-1)}{a\|X_{i}\|}.$$
Hence,
$$ L_{i}^{*} = \left( 1 - \frac{s'}{p}\right)\mathbf{P}\left( \frac{X_{i}^{T}\sigma\xi}{\|X_{i} \|} > t(s',X_{i})  \right) + \frac{s'}{p}\mathbf{P}\left( - \frac{X_{i}^{T}\sigma\xi}{\| X_{i}\|}>a\| X_{i}\| - t(s',X_{i}) \right).  $$
where $\xi$ is a standard Gaussian random vector in ${\mathbb R}^n$ independent of $X_i$. 
Notice now that $\varepsilon: = \frac{X_{i}^{T}\xi}{\|X_{i} \|} $ is a standard Gaussian random variable and it is independent  of $\|X_{i} \|$ since $X_{i} \sim {\mathcal N}(0, {\mathbb I}_n)$. Combining the above arguments we find that
$$
\inf_{\hat T\in [0,1]^p}\mathbb{E}_{\pi}\mathbf{E}_\eta \sum_{i=1}^p | {\hat T}_i(X,Y) - \eta_i |  \geq \psi_{+}(n,p,s',a,\sigma).$$
We conclude the proof by using the fact that the function $u\to \frac{\psi_{+}(n,p,u,a,\sigma)}{u}$ is decreasing for $u>0$ (cf. \cite{butucea}), so that 
$\psi_{+}(n,p,s',a,\sigma) \geq \frac{s'}{s}\psi_{+}(n,p,s,a,\sigma). $
\end{proof}

\begin{proof}[\textbf{Proof of Theorem \ref{thm:nec1}}]

In view of Theorem \ref{th:inf} with  $s'=s/2$, it is sufficient to bound $\psi_{+}=\psi_{+}(n,p,s,a,\sigma)$ from below. We have
$$  \psi_{+} \geq \left(p-s\right) \mathbf{P}\left( \sigma \varepsilon \geq t\left( \zeta \right) \right).  $$
We will use the following bound for the tails of standard Gaussian distribution:  For some $c'>0$,
 $$  \forall y \geq 2/3, \quad  \quad  \mathbf{P}\left( \varepsilon \geq y\right) \geq c' \frac{\exp(-{y^{2}}/{2})}{y}. $$
We also recall that the density $f_{n}$ of a chi-squared distribution with $n$ degrees of freedom has the form
\begin{equation}\label{20}
  f_{n}(t) = b_{n}t^{\frac{n}{2}-1}e^{-\frac{t}{2}}, \quad  t>0,
  \end{equation}
and $\underset{n \to \infty}{\lim} \frac{b_{n+1}}{b_{n}}\sqrt{n+1}=1,  $ so that for some $c''>0$ we have
$$ \forall n\geq 1,\qquad b_{n+1} \geq c'' \frac{b_{n}}{\sqrt{n+1}}. $$
Combining the above remarks we get 
\begin{eqnarray}
\psi_{+}&\geq & (p-s)\int_{2/3}^{\infty} \mathbf{P}\left(  \varepsilon \geq \left( \frac{\sqrt{u}a}{2\sigma} + \frac{\sigma \log\left( \frac{p}{s}-1\right)}{a\sqrt{u}} \right) \right) f_{n}(u)\mathrm{d}u\\
\nonumber
&\geq & c'\left( p-s \right) \int_{2/3}^{\infty}\frac{ \exp\Big(-\frac{1}{2}\Big( \frac{\sqrt{u}a}{2\sigma} + \frac{\sigma \log\left( \frac{p}{s}-1\right)}{a\sqrt{u}} \Big)^{2}\Big)}{\frac{\sqrt{u}a}{2\sigma} + \frac{\sigma \log\left( \frac{p}{s}-1\right)}{a\sqrt{u}}} f_{n}(u)\mathrm{d}u\\
\nonumber
&\geq & \frac{b_{n-1}c}{\sqrt{n}}\sqrt{s\left(p-s\right)} \int_{2/3}^{\infty}\frac{ \exp\Big(- \frac{u}{2}\left( 1+\frac{a^{2}}{4\sigma^{2}}\right) - \frac{\sigma^{2} \log\left( \frac{p}{s}-1\right)^{2}}{2a^{2}u} \Big)}{\frac{\sqrt{u}a}{2\sigma} + \frac{\sigma \log\left( \frac{p}{s}-1\right)}{a\sqrt{u}}} u^{\frac{n}{2}-1}\mathrm{d}u,
 \end{eqnarray}
 where $c=c' c''$. Set
 $$
 B= \int_{1}^{\infty}\frac{ \exp\Big(- \frac{v}{2} - \frac{\sigma^{2} \log\left( \frac{p}{s}-1\right)^{2}\left( 1+\frac{a^{2}}{4\sigma^{2}}\right)}{2a^{2}v} \Big)}{ 1 + \frac{\sigma^{2} \log\left( \frac{p}{s}-1\right)\left(1+\frac{a^{2}}{4\sigma^{2}}\right)}{a^{2}v}} v^{\frac{n-1}{2}-1}\mathrm{d}v.
 $$
 Using the change of variable $ v = u\left( 1+ \frac{a^{2}}{4\sigma^{2}}\right) $ and the assumptions of  the theorem we get
 \begin{eqnarray*}
\psi_{+} &\geq & \frac{c b_{n-1}}{\sqrt{n}}\sqrt{s\left(p-s\right)} \, e^{-\frac{n}{2}\log\left(1+\frac{a^{2}}{4\sigma^{2}}\right)} 
\int_{\frac23\left(1+\frac{a^{2}}{4\sigma^{2}}\right)}^{\infty}
\frac{ \exp\Big(-\ \frac{v}{2} - \frac{\sigma^{2} \log\left( \frac{p}{s}-1\right)^{2}\left( 1+\frac{a^{2}}{4\sigma^{2}}\right)}{2a^{2}v} \Big)}{\frac{\sqrt{v}a}{2\sigma\sqrt{1+\frac{a^{2}}{4\sigma^{2}}}} 
+ \frac{\sigma \log\left( \frac{p}{s}-1\right)\sqrt{1+\frac{a^{2}}{4\sigma^{2}}}}{a\sqrt{v}}} v^{\frac{n}{2}-1}\mathrm{d}v  \nonumber\\
& \geq   & c b_{n-1} B \sqrt{\frac{4\sigma^{2}\left( 1+\frac{a^{2}}{4\sigma^{2}}\right)}{na^{2}}}\sqrt{s\left(p-s\right)} e^{-\frac{n}{2}\log\left(1+\frac{a^{2}}{4\sigma^{2}}\right)}\\
\nonumber
& \geq   & c b_{n-1}B\sqrt{\frac{s\left(p-s\right)}{n\log\left( 1+\frac{a^{2}}{4\sigma^{2}}\right)}} e^{-\frac{n}{2}\log\left(1+\frac{a^{2}}{4\sigma^{2}}\right)} ,
 \end{eqnarray*}
where the second inequality uses the condition $a\leq \sqrt{2}\sigma$ to guarantee that
 $ \frac{2}{3}\left(1+\frac{a^{2}}{4\sigma^{2}}\right) \leq 1, $
 while the last inequality uses the fact that
 $(1+x) \log (1+x) \geq x,\ \forall x \geq 0. $
 To finish the proof, we need to bound $ b_{n-1}B$ from below.  
 We have
 \begin{eqnarray*}
B
&\geq & \int_{n}^{\infty}\frac{ \exp\Big(- \frac{v}{2} - \frac{\sigma^{2} \log\left( \frac{p}{s}-1\right)^{2}\left( 1+\frac{a^{2}}{4\sigma^{2}}\right)}{2a^{2}v} \Big)}{ 1 + \frac{\sigma^{2} \log\left( \frac{p}{s}-1\right)\left(1+\frac{a^{2}}{4\sigma^{2}}\right)}{a^{2}v}} v^{\frac{n-1}{2}-1}\mathrm{d}v \\
\nonumber 
& \geq & \frac{ \exp\Big( - \frac{\sigma^{2} \log\left( \frac{p}{s}-1\right)^{2}\left( 1+\frac{a^{2}}{4\sigma^{2}}\right)}{2a^{2}n} \Big)}{ 1 + \frac{\sigma^{2} \log\left( \frac{p}{s}-1\right)\left(1+\frac{a^{2}}{4\sigma^{2}}\right)}{a^{2}n}}\frac{\int_{n}^{\infty}f_{n-1}(u)\mathrm{d}u}{b_{n-1}}.
 \end{eqnarray*}
 The last inequality is due to the fact that the function $x \to \frac{e^{-\frac{c}{x}}}{1+\frac{1}{x}}$ is increasing for $x>0$, for any fixed $c>0$.  Since $n>\frac{2\sigma^{2}\log\left( \frac{p}{s}-1\right)}{a^{2}}$ and $a^{2}<2\sigma^{2}$, we deduce from the last display that
 $$  b_{n-1}B  \geq 
  \frac{4}{7} \exp\Big(-\frac{3}{8}\log{\left( \frac{p}{s}-1\right)}\Big) \int_{n}^{\infty}f_{n-1}(u)\mathrm{d}u.  $$
 {Proposition 3.1} from \cite{ing} implies that, for some absolute constant $c>0$,
 $$ \int_{n}^{\infty}f_{n-1}(u)\mathrm{d}u > c$$
(indeed, $n$ is very close to the median of a chi-squared random variable with $n-1$ degrees of freedom). Combining the above inequalities we obtain
 $$  \psi_{+}\geq C\sqrt{\frac{s^{7/4}(p-s)^{1/4}}{n\log\left(1+\frac{a^{2}}{4\sigma^{2}}\right)}}e^{-\frac{n}{2}\log\left(1+\frac{a^{2}}{4\sigma^{2}}\right)} . $$

\end{proof}
\begin{proof}[\textbf{Proof of Theorem \ref{thm:suf}}]
In view of {Theorem \ref{thm:suf1}}, it is sufficient to bound from above the expression
$$  \psi\left(n,p,s,a,\sigma \right) = \left(p-s\right) \mathbf{P}\left( \sigma \varepsilon \geq t\left( \zeta  \right) \right)+s\mathbf{P}\left( \sigma \varepsilon \geq \left(a\|\zeta\|-t\left( \zeta \right) \right)_{+}\right).  $$
Introducing the event $\mathbb{D} = \left\{ a\| \zeta\| \geq t(\zeta) \right\}$ we get
$$\mathbf{P}\left( \sigma \varepsilon \geq \left(a\|\zeta\|-t\left( \zeta \right) \right)_{+}\right) \leq \mathbf{P}\left( \{ \sigma \varepsilon \geq a\|\zeta\|-t\left( \zeta \right) \} \cap \mathbb{D}\right) +\frac{1}{2} \mathbf{P}\left( \mathbb{D}^{c}\right).$$
Using the assumption on $n_{2}$ we obtain
$$  \mathbf{P}\left( \mathbb{D}^{c} \right) = \mathbf{P}\left( \| \zeta\|^{2} \leq \frac{2\sigma^{2} \log(\frac{p}{s}-1)   }{a^{2}} \right) \leq \mathbf{P}\left( \| \zeta\|^{2} \leq \frac{n_{2}}{2} \right) . $$
Here, $\| \zeta\|^{2}$ is a chi-squared random variable with $n_{2}$ degrees of freedom. Lemma \ref{lem:chi2}
implies
$$ \frac{1}{2}\mathbf{P}\left( \mathbb{D}^{c} \right)  \leq e^{-\frac{n_{2}}{24}}. $$
Thus, to finish the proof it remains to show that
$$ \left(p-s\right) \mathbf{P}\left( \sigma \varepsilon \geq t\left( \zeta  \right) \right)+s\mathbf{P}\left(\{ \sigma \varepsilon \geq a\|\zeta\|-t\left( \zeta \right)  \} \cap \mathbb{D}\right) \leq 2\sqrt{s(p-s)}e^{-\frac{n_{2}}{2}\log\left(1+\frac{a^{2}}{4\sigma^{2}}\right)}. $$
The bound $\mathbf{P}\left( \varepsilon \geq y\right) \leq e^{-\frac{y^{2}}{2}}, \ \forall y > 0,$ on the tail of standard Gaussian distribution yields
\begin{eqnarray*}
\left(p-s\right) \mathbf{P}\left( \sigma \varepsilon \geq t\left( \zeta  \right) \right) &\leq & \left( p-s \right) \int_{0}^{\infty}\frac{ e^{-\frac{1}{2}\left( \frac{\sqrt{u}a}{2\sigma} + \frac{\sigma \log\left( \frac{p}{s}-1\right)}{a\sqrt{u}} \right)^{2}}}{1+\frac{\sqrt{u}a}{2\sigma} + \frac{\sigma \log\left( \frac{p}{s}-1\right)}{a\sqrt{u}}} f_{n_{2}}(u)\mathrm{d}u\\
\nonumber
&\leq & b_{n_{2}}\sqrt{s\left(p-s\right)} \int_{0}^{\infty}e^{-\ \frac{u}{2}\left( 1+\frac{a^{2}}{4\sigma^{2}}\right) } u^{\frac{n_{2}}{2}-1}\mathrm{d}u,
 \end{eqnarray*}
 where $f_{n_{2}}(\cdot)$ is the density of chi-squared distribution with $n_{2}$ degrees of freedom
 and $b_{n_{2}}$ is the corresponding normalizing constant, cf. \eqref{20}.
Using again the bound $\mathbf{P}\left( \varepsilon \geq y\right) \leq e^{-\frac{y^{2}}{2}}, \ \forall y > 0,$ and the inequality
$$  \frac{\sqrt{u}a}{2\sigma} - \frac{\sigma \log\left( \frac{p}{s}-1\right)}{a\sqrt{u}} \geq 0,\qquad \forall u \geq \frac{2\sigma^{2}\log\left(\frac{p}{s}-1\right)}{a^{2}}, $$
we get 
\begin{eqnarray*}
s\mathbf{P}\left( \{ \sigma \varepsilon \geq a\|\zeta\|-t\left( \zeta \right) \} \cap \mathbb{D}\right) &\leq & s \int_{\frac{2\sigma^{2}\log\left(\frac{p}{s}-1\right)}{a^{2}}}^{\infty} e^{-\frac{1}{2}\left( \frac{\sqrt{u}a}{2\sigma} - \frac{\sigma \log\left( \frac{p}{s}-1\right)}{a\sqrt{u}} \right)^{2}}
 f_{n_{2}}(u)\mathrm{d}u\\
\nonumber
&\leq & b_{n_{2}}\sqrt{s\left(p-s\right)} \int_{0}^{\infty}e^{-\ \frac{u}{2}\left( 1+\frac{a^{2}}{4\sigma^{2}}\right) } u^{\frac{n_{2}}{2}-1}\mathrm{d}u.
 \end{eqnarray*}
The change of variable, $v = u\left(1+\frac{a^{2}}{4\sigma^{2}}\right)$ yields
 \begin{eqnarray*}
b_{n_{2}} \int_{0}^{\infty}e^{- \frac{u}{2}\left( 1+\frac{a^{2}}{4\sigma^{2}}\right) } u^{\frac{n_{2}}{2}-1}\mathrm{d}u & = & b_{n_{2}}e^{- \frac{n_{2}}{2}\log\left(1+\frac{a^{2}}{4\sigma^{2}}\right)} \int_{0}^{\infty}e^{- \frac{v}{2} } v^{\frac{n_{2}}{2}-1}\mathrm{d}v  \\
\nonumber
& =   & e^{- \frac{n_{2}}{2}\log\left(1+\frac{a^{2}}{4\sigma^{2}}\right)}. 
 \end{eqnarray*}
That concludes the proof.
\end{proof}

\begin{proof}[\textbf{Proof of Lemma \ref{lem:ck}}]
Recall that the density of a Student random variable $Z$ with $k$ degrees of freedom is given by:
$$f_{Z}(t) = c_{k}^* \left(1+\frac{t^{2}}{k}\right)^{-\frac{k+1}{2}}, \quad t\in{\mathbb R},$$
where $c_{k}^*>0$ satisfies
\begin{equation}\label{eq:ck}
\underset{k \to \infty}{\lim}c_{k}^* = \sqrt{2\pi}. 
\end{equation}
Define, for $t>0$, 
$$
g(t)= -c_{k}^*t^{-1}\left(1+\frac{t^{2}}{k}\right)^{-\frac{k-1}{2}}.
$$
It is easy to check that the derivative of $g$ has the form
$$g'(t) = \left(1+\frac{1}{t^{2}}\right)f_{Z}(t).$$
Hence, for all $b\ge  1/\sqrt{k}$,
$$  -2g(\sqrt{k}b) =2\int_{\sqrt{k}b}^{\infty}g'(t)\mathrm{d}t \leq \mathbf{P}\left(|Z|\ge\sqrt{k}b\right) \leq  4\int_{\sqrt{k}b}^{\infty}g'(t)\mathrm{d}t = -4g(\sqrt{k}b).$$
The lemma follows since, 
in view of (\ref{eq:ck}), there exist two positive constants $c$ and $C$ such that
 $c \leq c_{k}^* \leq C$ for all $  k\geq1$.
\end{proof}

\begin{proof} [\textbf{Proof of Lemma \ref{lem:3}}]
It is not hard to check that the random variable $\frac{\left|u^{\top}V\right|}{\|u\|}$ is $\sigma$-sub-Gaussian for any fixed $u\in\mathbb{R}^{n}$. Also, any $\sigma$-sub-Gaussian random $\zeta$ variable satisfies $\mathbf{P}(|\zeta|\ge t)\le 2e^{-\frac{t^2}{2\sigma^2}}$ for all $t>0$. Therefore, we have the following bound for the conditional probability:
$$
\mathbf{P}\left( \frac{\left|U^{\top}V\right|}{\|U\|} \geq t \|U\| \Big \vert \ U \right) \le 2e^{-\frac{t^2\|U\|^2}{2\sigma^2}}, \quad \forall \ t>0. 
$$
This implies
\begin{align}\nonumber
\mathbf{P}\left( \frac{\left|U^{\top}V\right|}{\|U\|^{2}} \geq t \right) &\le 2\mathbf{E}\Big[e^{-\frac{t^2\|U\|^2}{2\sigma^2}}\mathbf{1}\left\{ \| U\| \geq \sqrt{n}/{2} \right\}\Big] + \mathbf{P}\left( \| U \| \leq \sqrt{n}/{2} \right)\\
&\le 2e^{-\frac{nt^2}{8\sigma^2}} + \mathbf{P}\left( \| U \| \leq \sqrt{n}/{2} \right).
\label{eq1}
\end{align}
To bound the last probability, we apply the following inequality  \cite[Proposition 2.6]{wegkamp}.
\begin{lem}\label{lem:weg}
Let $Z_1, Z_2,\dots , Z_n$ be independent, nonnegative random variables with $\mathbf{E}(Z_i) = \mu_i$ and $\mathbf{E}(Z_i^2) \le v^2 $. Then, for all $x>0$,
$$
\mathbf{P}\Big(\frac1n\sum_{i=1}^n (Z_i-\mu_i) \le -x\Big)\le e^{-\frac{nx^2}{2v^2}}.
$$
 \end{lem}
Using this lemma with $Z_i=U_i^2$, $\mu_i\equiv 1$, $x=3/4$, and $v^2=\sigma_1^4$ we find
$$
\mathbf{P}\left( \| U \| \leq \sqrt{n}/{2} \right) \le e^{-\frac{9n}{32 \sigma_1^4}},
$$
which together with \eqref{eq1} proves the lemma.
\end{proof}

\begin{proof} [\textbf{Proof of Proposition \ref{prop:subGaussian}}] Under the assumptions of the proposition,  the columns of matrix $X$ have the covariance matrix $\mathbb{I}_{p}$. Without loss of generality, we may assume that this covariance matrix is $\frac{1}{2} \mathbb{I}_{p}$ and replace $\sigma$ by $\frac{\sigma}{\sqrt{2}}$. We next define the event
$$\mathbb{A} = \{ \text{the design matrix } X \text{ satisfies the $WRE(s,20)$ condition}\},$$
where the $WRE$ condition is defined in \cite{bellec}.
It is easy to check that the assumptions of {Theorem 8.3} in \cite{bellec} are fulfilled, with $\Sigma = \frac{1}{2}\mathbb{I}_{p}$, $\kappa = \frac{1}{2}$ and  $ n_{1} \geq C_{0}s\log(2p/s)$ for some $C_{0}>0$ large enough. Using {Theorem 8.3} in \cite{bellec} we get 
$$ \mathbf{P}\left( \mathbb{A}^{c} \right) \leq 3 e^{-C's\log{2p/s}}, $$
for some $C'>0$.
Now, in order to prove the proposition, we use the bound
$$ \mathbf{P}_{\beta}\left( \|\hat{\beta} - \beta\|^{2} \geq \sigma^{2} \delta^{2} \right) \leq \mathbf{P}_{\beta}\left( \left\{ \|\hat{\beta} - \beta\|^{2} \geq \sigma^{2} \delta^{2} \right\} \cap \mathbb{A} \right) + \mathbf{P}\left( \mathbb{A}^{c} \right) .   $$
Under the assumption $ n_{1} \geq C_{0}s\log(ep/s)/\delta^2$, we have 
$$  \mathbf{P}_{\beta}\left( \left\{ \|\hat{\beta} - \beta\|^{2} \geq \sigma^{2} \delta^{2} \right\} \cap \mathbb{A} \right) \leq \mathbf{P}_{\beta}\left( \left\{ \|\hat{\beta} - \beta\|^{2} \geq C_{0}\sigma^{2} \frac{s\log{ep/s}}{n_{1}}\right\} \cap \mathbb{A} \right). $$
By choosing $C_{0}$ large enough, and using {Proposition 4} from \cite{collier2018} we get that, for some $C''>0$,
$$\mathbf{P}_{\beta}\left( \left\{ \|\hat{\beta} - \beta\|^{2} \geq \sigma^{2} \delta^{2} \right\} \cap \mathbb{A} \right) \leq C''\left(e^{-s\log(2p/s)/C''} + e^{-n_{1}/C''}\right).$$
Recalling that $n_{1} \geq C_0s\log(2p/s)$ and combining the above inequalities we obtain the result of the proposition with $C_{1} = 2C''+3$ and $C_{2}=C' \wedge 1/C'' \wedge C_0/C''$.
\end{proof}
\begin{lem}\label{lem:MOM_SLOPE}
Let $\hat{\beta}^{*}$ be the MOM-SLOPE estimator of \cite{lec17}. Let $X$ and $\xi$ satisfy the 
conditions of Section 6. Then, $\hat{\beta}^{*}$ satisfies Assumption~\ref{MOM-SLOPE}.
\end{lem}
\begin{proof} [\textbf{Proof of Lemma \ref{lem:MOM_SLOPE}}] We apply Theorem 6 in \cite{lec17}. Thus, it is enough to check that items 1-5 of Assumption 6 in \cite{lec17} are satisfied. Item 1 is immediate since $|I| = n_{1} - |O| \geq n_{1}/2$, and $|O| \leq c_{0}s\log(ep/s)$. To check item 2, we  first note that the random variable ${\bf x}^{\top}_1t$ is $ \| t \|\sigma_X$-sub-Gaussian for any $t \in \mathbb{R}^{p}$. It follows from the standard properties of sub-Gaussian random variables \cite[Lemma 5.5]{vershynin} that,
 for some $C>0$,
$$ \big(\mathbf{E}|{\bf x}^{\top}_1t |^{d} \big)^{1/d} \leq  C \| t \| \sqrt{d},\quad  \forall t \in \mathbb{R}^{p}, \forall d \geq 1.$$
On the other hand, since the elements of ${\bf x}_1$ are centered random variables with variance 1,
\begin{equation}\label{iso}
 \big(\mathbf{E}|{\bf x}^{\top}_1t |^{2} \big)^{1/2} =  \| t \| , \quad  \forall t \in \mathbb{R}^{p}.
 \end{equation}
Combining the last two displays proves item 2. 
Item 3 holds since we assume that $\mathbf{E}(|\xi_{i}|^{q_0}) \leq \sigma^{q_0},  i \in I,$ with $q_0=2+q$. 
To prove item 4, we use \eqref{iso} and the fact that, for some $C>0$,
$$\mathbf{E}|{\bf x}^{\top}_1t |  \ge C  \| t \| , \quad  \forall t \in \mathbb{R}^{p},$$
due to Marcinkiewicz-Zygmund inequality \cite[page 82]{petrov}.
Finally we have that, for some $c>0$,
$$ {\rm Var}( \xi_{1} {\bf x}^{\top}_1t) \leq \mathbf{E}[\xi_{1}^{2}] \,\mathbf{E}\big[ ({\bf x}^{\top}_1t)^{2} \big] \leq c \mathbf{E}\big[ ({\bf x}^{\top}_1t)^{2} \big], \quad  \forall t \in \mathbb{R}^{p}.$$
Thus, all conditions of Theorem 6 in \cite{lec17} are satisfied. Application of this theorem yields the result.
\end{proof}

\begin{proof}[\textbf{Proof of Lemma \ref{lem:7}}]
We first prove that for all $i \in I$ and $1\leq j \leq p$,  
\begin{equation}\label{ebig}
\mathbf{E}_{\beta}\big( 
            (\varepsilon_{j}^{(i)})^{2} \mathbf{1} \left\{ \mathbb{A}_* \right\} 
        \big) 
    \leq C\frac{K\sigma^{2}}{n} ,
    \end{equation}
    where $C>0$ depends only on the sub-Gaussian constant $\sigma_X$.
Indeed, 
the components of $ \varepsilon^{(i)}$ have the form
$$ \varepsilon^{(i)}_{j} = \left( \frac{1}{q}\|{\bf X}^{(i)}_{j}\|^{2} - 1\right)(\beta_{j} - \hat{\beta}_{j}^*) + \frac{1}{q} \sum_{k \neq j} {\bf X}^{(i)\top}_{j}{\bf X}_{k}^{(i)} (\beta_{k} - \hat{\beta}_{k}^*) + \frac{1}{q}  {\bf X}_{j}^{(i)\top}\xi,$$
where ${\bf X}^{(i)}_{j}$ is the $j$th column of ${\bf X}^{(i)}$.
Conditioning first on ${\bf X}^{(i)}_{j}$, we get
\begin{align*}\mathbf{E}_{\beta}\big( 
            (\varepsilon_{j}^{(i)})^{2} \mathbf{1} \{ \mathbb{A}_*\} 
        \big)  
& \leq 2
        \mathbf{E}_{\beta}\left[
            \left( \|\hat{\beta}^* - \beta\|^{2}\left( \frac{1}{q}\|{\bf X}^{(i)}_{j}\|^{2} - 1 \right)^{2} 
            + \frac{1}{q^{2}}\left( \|\hat{\beta}^* - \beta\|^{2} + \sigma^{2} \right) \| {\bf X}_{j}^{(i)}\|^{2} \right)
            \mathbf{1}\{\mathbb{A}_*\} 
        \right]
        \\
 & \leq 2\sigma^{2}
        \mathbf{E}\left[
            \left( \frac{1}{q}\|{\bf X}^{(i)}_{j}\|^{2} - 1 \right)^{2} 
            + \frac{2}{q^{2}}\| {\bf X}_{j}^{(i)}\|^{2} 
        \right].        
        \end{align*}
         Since $\mathbf{E}(X_{kl}^{2}) =1$ for all $k$ and $l$, we have 
        $\mathbf{E}\|{\bf X}^{(i)}_{j}\|^{2}=q$. Furthermore, $\mathbf{E}( X_{kl}^{4}) \leq \bar C$ where $\bar C$ depends only on the sub-Gaussian constant $\sigma_X$. Using these remarks we obtain from the last display that
        $$\mathbf{E}_{\beta}\big( 
            (\varepsilon_{j}^{(i)})^{2} \mathbf{1} \{ \mathbb{A}_*\} 
        \big)  
        \leq \frac{2(\bar C+2)\sigma^{2}}{q}. $$      
As $q=\lfloor n/K \rfloor$ this yields \eqref{ebig}. 

Next,  the definition of the median immediately implies that
  $$ \left\{ |Med(\varepsilon_{j})| \geq t \right\} \subseteq  \left\{ \sum_{i=1}^{K} \mathbf{1}_{\{|\varepsilon_{j}^{(i)}| \geq t \}} \geq \frac{K}{2} \right\}, \quad \forall t>0. $$
It follows that
  \begin{align*}
      \mathbf{P}_{\beta}\left( |Med(\varepsilon_{j})| \geq t \right) &\leq \mathbf{P}_{\beta}\left( \left\{ |Med(\varepsilon_{j})| \geq t \right\} \cap \mathbb{A_*} \right) + \mathbf{P}(\mathbb{A}_*^{c}) \\
      &\leq  \mathbf{P}_{\beta}\left( \left\{  \sum_{i=1}^{K} \mathbf{1}_{\{|\varepsilon_{j}^{(i)}| \geq t \}} \geq \frac{K}{2} \right\} \cap \mathbb{A}_* \right) + \mathbf{P}_{\beta}(\mathbb{A}_*^{c}) \\
      & \leq \mathbf{P}_{\beta}\left( \sum_{i=1}^{K} \mathbf{1}_{\{|\varepsilon_{j}^{(i)}| \geq t \} \cap \mathbb{A}_*} \geq \frac{K}{2}  \right) + \mathbf{P}_{\beta}(\mathbb{A}_*^{c}). 
  \end{align*}
Since the number of outliers $|O|$ does not exceed $\lfloor K/4\rfloor$  there are at least $K':= K-\lfloor K/4\rfloor$ blocks that contain only observations from $I$. Without loss of generality,  assume that these  blocks are indexed by $1,\dots, K'$. Hence 
 \begin{eqnarray}
 \mathbf{P}_{\beta} ( |Med(\varepsilon_{j})| \geq t)  \leq \mathbf{P}_{\beta}\bigg( \sum_{i=1}^{K'} \mathbf{1}_{\{|\varepsilon_{j}^{(i)}| \geq t \} \cap \mathbb{A}_*} \geq \frac{K}{4}  \bigg) + \mathbf{P}_{\beta}(\mathbb{A}_*^{c}).  
 \label{mom1}
\end{eqnarray}
Note that using \eqref{ebig}  we have, for all $i=1,\dots,K'$,
$$ \mathbf{P}_{\beta}\left( \{|\varepsilon_{j}^{(i)}| \geq t \} \cap \mathbb{A}_* \right) \leq
\mathbf{E}_{\beta}\big( 
            (\varepsilon_{j}^{(i)})^{2} \mathbf{1} \{ \mathbb{A}_*\} 
        \big)  /t^2
  \leq \frac{CK\sigma^{2}}{t^{2}n} \leq \frac{1}{5}. $$
The last inequality is granted by a choice of large enough constant $c_4$ in the definition of $t$. Thus, introducing the notation $\zeta_i = \mathbf{1}_{\{|\varepsilon_{j}^{(i)}| \geq t \} \cap \mathbb{A}_*}$ we obtain
\begin{eqnarray}\nonumber
\mathbf{P}_{\beta}\bigg( \sum_{i=1}^{K'} \mathbf{1}_{\{|\varepsilon_{j}^{(i)}| \geq t \} \cap \mathbb{A}_*} \geq \frac{K}{4}  \bigg)&\le& 
\mathbf{P}_{\beta}\bigg( \sum_{i=1}^{K'} (\zeta_i- \mathbf{E}_{\beta}(\zeta_i)) \geq \frac{K}{4} - \frac{K'}{5}  \bigg)\\
&\le & \mathbf{P}_{\beta}\bigg( \sum_{i=1}^{K'} (\zeta_i- \mathbf{E}_{\beta}(\zeta_i)) \geq \frac{K}{20}   \bigg)\leq e^{-K/400}\label{mom2}
\end{eqnarray}
where the last inequality is an application of Hoeffding's inequality. Combining \eqref{mom1} and \eqref{mom2} proves the lemma.

\end{proof}

\newpage

\bibliographystyle{unsrt}

\begin{thebibliography}{10}

\bibitem{ZH}
Peng Zhao and Bin Yu.
\newblock On model selection consistency of {L}asso.
\newblock {\em The Journal of Machine Learning Research}, 7:2541--2563, 2006.

\bibitem{wainwrightaLasso}
Martin~J Wainwright.
\newblock Sharp thresholds for high-dimensional and noisy sparsity recovery
  using-constrained quadratic programming ({L}asso).
\newblock {\em IEEE Transactions on Information Theory}, 55(5):2183--2202,
  2009.

\bibitem{tropp}
Joel~A Tropp and Anna~C Gilbert.
\newblock Signal recovery from random measurements via orthogonal matching
  pursuit.
\newblock {\em IEEE Transactions on Information Theory}, 53(12):4655--4666,
  2007.

\bibitem{zhang}
Tong Zhang.
\newblock Sparse recovery with orthogonal matching pursuit under {RIP}.
\newblock {\em IEEE Transactions on Information Theory}, 57(9):6215--6221,
  2011.

\bibitem{caiOMP}
T~Tony Cai and Lie Wang.
\newblock Orthogonal matching pursuit for sparse signal recovery with noise.
\newblock {\em IEEE Transactions on Information Theory}, 57(7):4680--4688,
  2011.

\bibitem{fletcher}
Alyson~K Fletcher, Sundeep Rangan, and Vivek~K Goyal.
\newblock Necessary and sufficient conditions for sparsity pattern recovery.
\newblock {\em IEEE Transactions on Information Theory}, 55(12):5758--5772,
  2009.

\bibitem{joseph}
Antony Joseph.
\newblock Variable selection in high-dimension with random designs and
  orthogonal matching pursuit.
\newblock {\em Journal of Machine Learning Research}, 14(1):1771--1800, 2013.

\bibitem{wainwrightb}
Martin~J Wainwright.
\newblock Information-theoretic limits on sparsity recovery in the
  high-dimensional and noisy setting.
\newblock {\em IEEE Transactions on Information Theory}, 55(12):5728--5741,
  2009.

\bibitem{aeron}
Shuchin Aeron, Venkatesh Saligrama, and Manqi Zhao.
\newblock Information theoretic bounds for compressed sensing.
\newblock {\em IEEE Transactions on Information Theory}, 56(10):5111--5130,
  2010.

\bibitem{saligrama1}
Venkatesh Saligrama and Manqi Zhao.
\newblock Thresholded basis pursuit: {LP} algorithm for order-wise optimal
  support recovery for sparse and approximately sparse signals from noisy
  random measurements.
\newblock {\em IEEE Transactions on Information Theory}, 57(3):1567--1586,
  2011.

\bibitem{wangWain}
Wei Wang, Martin~J Wainwright, and Kannan Ramchandran.
\newblock Information-theoretic limits on sparse signal recovery: Dense versus
  sparse measurement matrices.
\newblock {\em IEEE Transactions on Information Theory}, 56(6):2967--2979,
  2010.

\bibitem{rad}
Kamiar~Rahnama Rad.
\newblock Nearly sharp sufficient conditions on exact sparsity pattern
  recovery.
\newblock {\em IEEE Transactions on Information Theory}, 57(7):4672--4679,
  2011.

\bibitem{saligrama3}
David Gamarnik and Ilias Zadik.
\newblock Sparse high-dimensional linear regression. {A}lgorithmic barriers and
  a local search algorithm.
\newblock {\em arXiv preprint arXiv:1711.04952}, 2017.

\bibitem{jin1}
Pengsheng Ji and Jiashun Jin.
\newblock {UPS} delivers optimal phase diagram in high-dimensional variable
  selection.
\newblock {\em Annals of Statistics}, 40(1):73--103, 2012.

\bibitem{jin2}
Jiashun Jin, Cun-Hui Zhang, and Qi~Zhang.
\newblock Optimality of graphlet screening in high dimensional variable
  selection.
\newblock {\em Journal of Machine Learning Research}, 15(1):2723--2772, 2014.

\bibitem{jin3}
Tracy Ke, Jiashun Jin, and Jianqing Fan.
\newblock Covariance assisted screening and estimation.
\newblock {\em Annals of Statistics}, 42(6):2202, 2014.

\bibitem{aksoylar}
Cem Aksoylar, George~K Atia, and Venkatesh Saligrama.
\newblock Sparse signal processing with linear and nonlinear observations: A
  unified {S}hannon-theoretic approach.
\newblock {\em IEEE Transactions on Information Theory}, 63(2):749--776, 2017.

\bibitem{deru}
Alexis Derumigny.
\newblock Improved bounds for square-root lasso and square-root slope.
\newblock {\em Electronic Journal of Statistics}, 12(1):741--766, 2018.

\bibitem{candes}
Malgorzata Bogdan, Ewout Van Den~Berg, Chiara Sabatti, Weijie Su, and
  Emmanuel~J Cand{\`e}s.
\newblock Slope-adaptive variable selection via convex optimization.
\newblock {\em The annals of applied statistics}, 9(3):1103, 2015.

\bibitem{butucea}
Cristina Butucea, Mohamed Ndaoud, Natalia~A Stepanova, and Alexandre~B
  Tsybakov.
\newblock Variable selection with {H}amming loss.
\newblock {\em Annals of Statistics}, 46(5):1837--1875, 2018.

\bibitem{bellec}
Pierre~C Bellec, Guillaume Lecu{\'e}, and Alexandre~B Tsybakov.
\newblock Slope meets {L}asso: improved oracle bounds and optimality.
\newblock {\em arXiv preprint arXiv:1605.08651}, 2016.

\bibitem{Adel}
Adel Javanmard and Andrea Montanari.
\newblock De-biasing the {L}asso: {O}ptimal sample size for gaussian designs.
\newblock {\em arXiv preprint arXiv:1508.02757}, 2015.

\bibitem{wang17}
Shuaiwen Wang, Haolei Weng, and Arian Maleki.
\newblock Which bridge estimator is optimal for variable selection?
\newblock {\em arXiv preprint arXiv:1705.08617}, 2017.

\bibitem{cgpt}
L.~Cavalier, G.K. Golubev, D.~Picard, and A.B. Tsybakov.
\newblock Oracle inequalities for inverse problems.
\newblock {\em Annals of Statistics}, 30(3):843--874, 2002.

\bibitem{lptv}
K.~Lounici, M.~Pontil, S.~van~de Geer, and A.B. Tsybakov.
\newblock Oracle inequalities and optimal inference under group sparsity.
\newblock {\em Annals of Statistics}, 39(5):2164--2204, 2011.

\bibitem{collier2018}
Olivier Collier, La{\"e}titia Comminges, and Alexandre~B Tsybakov.
\newblock Some effects in adaptive robust estimation under sparsity.
\newblock {\em arXiv preprint arXiv:1802.04230}, 2018.

\bibitem{MOM3}
Arkadii Nemirovskii and David~Borisovich Yudin.
\newblock {\em Problem Complexity and Method Efficiency in Optimization}.
\newblock Wiley, New York, 1983.

\bibitem{MOM2}
Mark~R Jerrum, Leslie~G Valiant, and Vijay~V Vazirani.
\newblock Random generation of combinatorial structures from a uniform
  distribution.
\newblock {\em Theoretical Computer Science}, 43:169--188, 1986.

\bibitem{MOM1}
Noga Alon, Yossi Matias, and Mario Szegedy.
\newblock The space complexity of approximating the frequency moments.
\newblock {\em Journal of Computer and System Sciences}, 58(1):137--147, 1999.

\bibitem{lec17}
Guillaume Lecu{\'e} and Matthieu Lerasle.
\newblock Robust machine learning by median-of-means: theory and practice.
\newblock {\em arXiv preprint arXiv:1711.10306}, 2017.

\bibitem{ing}
Tadeusz Inglot.
\newblock Inequalities for quantiles of the chi-square distribution.
\newblock {\em Probability and Mathematical Statistics}, 30(2):339--351, 2010.

\bibitem{wegkamp}
M~Wegkamp.
\newblock Model selection in nonparametric regression.
\newblock {\em Annals of Statistics}, 31(1):252--273, 2003.

\bibitem{vershynin}
R~Vershynin.
\newblock Introduction to the non-asymptotic analysis of random matrices.
\newblock {\em In: Compressed Sensing. Cambridge Univ. Press, Cambridge}, pages
  210--268, 2012.

\bibitem{petrov}
V~V Petrov.
\newblock {\em Limit Theorems of Probability Theory}.
\newblock Clarendon Press, Oxford, 1995.

\end{thebibliography}

\end{document}